\documentclass[12pt, reqno]{amsart}
\usepackage[foot]{amsaddr}
\usepackage[all,tips]{xy}
\usepackage{latexsym, amsfonts, amsmath, amssymb, mathrsfs, graphicx, xcolor, ytableau, hyperref, float, xurl}
\allowdisplaybreaks
\usepackage{centernot}
\usepackage{tikz}

\definecolor{red}{rgb}{1,0,0}
\definecolor{magenta}{rgb}{1,0,1}
\definecolor{dartmouthgreen}{rgb}{0.05, 0.5, 0.06}
\definecolor{purple(x11)}{rgb}{0.63,0.36,0.94}
\definecolor{turquoise}{rgb}{0.25, 0.87, 0.81}
\newtheorem{theorem}{Theorem}[section]
\newtheorem{lemma}[theorem]{Lemma}
\newtheorem{proposition}[theorem]{Proposition}
\newtheorem{corollary}[theorem]{Corollary}
\newtheorem{conjecture}[theorem]{Conjecture}

\theoremstyle{definition}
\newtheorem{remark}[theorem]{Remark}

\newcommand{\Cal}[1]{\ensuremath{\mathcal{#1}}}

\newcommand{\lp}{\left(}
\newcommand{\rp}{\right)}

\def\C{{\mathbb C}}

\def\N{{\mathbb N}}
\def\Z{{\mathbb Z}}
\def\Q{{\mathbb Q}}

\def\O_K{{\Cal{O}_{K}}}
\def\O_F{{\Cal{O}_{F}}}
\def\N_F{{\Cal{N}_{F/\Q}}}

\def\O_K{{\Cal{O}_{K}}}
\def\O_F{{\Cal{O}_{F}}}
\def\N_F{{\Cal{N}_{F/\Q}}}

\DeclareMathOperator{\arm}{arm}
\DeclareMathOperator{\leg}{leg}
\DeclareMathOperator{\coarm}{coarm}
\DeclareMathOperator{\coleg}{coleg}

\numberwithin{equation}{section}
\numberwithin{theorem}{section}

\title{Inequalities and asymptotics for hook numbers in restricted partitions}

\author{William Craig}
\address{Department of Mathematics, United States Naval Academy, 572C Holloway Road
Mail Stop 9E. Annapolis, MD 21402}
\email{wcraig@usna.edu}
\thanks{The first author is the corresponding author.}

\author{Madeline Locus Dawsey}
\address{University of Texas at Tyler,
3900 University Boulevard, Tyler, TX 75799, USA}
\email{mdawsey@uttyler.edu}

\author{Guo--Niu Han}
\address{IRMA, Universit\'e de Strasbourg et CNRS, 7 rue Ren\'e Descartes, 67084 Strasbourg, France}
\email{guoniu.han@unistra.fr}

 
\begin{document} 
 
\maketitle

\begin{abstract}
    In this paper, we consider the asymptotics of hook numbers of partitions in restricted classes. More specifically, we compare the frequency with which partitions into odd parts and partitions into distinct parts have hook numbers equal to $h \geq 1$ by deriving an asymptotic formula for the total number of hooks equal to $h$ that appear among partitions into odd and distinct parts, respectively. We use these asymptotic formulas to prove a recent conjecture of the first author and collaborators that for $h \geq 2$ and $n \gg 0$, partitions into odd parts have, on average, more hooks equal to $h$ than do partitions into distinct parts. We also use our asymptotics to prove certain probabilistic statements about how hooks distribute in the rows of partitions.
\end{abstract}

\section{Introduction} \label{intro}

A {\it partition} $\lambda = \lp \lambda_1, \lambda_2, \dots, \lambda_\ell \rp$ of an integer $n \geq 0$ is a representation of $n$ in the form $n = \lambda_1 + \lambda_2 + \dots + \lambda_\ell$, where reorderings of this sum are considered identical. We use the standard notations $\lambda \vdash n$ or $\left| \lambda \right| = n$ to say that a partition $\lambda$ has {\it size} $n$, and we let $\ell := \ell\lp \lambda \rp$ denote the number of parts, or {\it length}, of $\lambda$. We also let $p(n)$ be the number of partitions of $n$. The set of all partitions is denoted $\mathcal{P}$. The theory of partitions, and the closely related theory of $q$-series, have a vast literature and intersect regularly with combinatorics, harmonic analysis, number theory, mathematical physics, and many other areas of mathematics. For an overview of the theory of partitions, see Andrews' seminal book \cite{Andrews}.

In this paper, we focus on the hook numbers of partitions. These are defined in terms of the {\it Ferrers diagram} of the partition $\lambda$, whereby $\lambda = \lp \lambda_1, \dots, \lambda_\ell \rp$ is represented as a diagram of left-adjusted rows of boxes in which the $i$th row contains $\lambda_i$ boxes. The {\it hook number} $h_{i,j}(\lambda)$ of the cell in the $i$th row and $j$th column of $\lambda$ is defined as the length of the $L$-shape formed by the boxes below and to the right of this box, including the box itself. We represent the multiset of all hook numbers of $\lambda$ as $\mathcal H\lp \lambda \rp$. We give an example below in Figure \ref{Figure} for the partition $4+3+2$.

\begin{figure}[H] \label{Figure}
\centering $$\begin{ytableau} 6&5&3&1 \\ 4&3&1 \\ 2&1  \end{ytableau}$$
\caption{{\small Hook numbers of the partition $\lambda=(4,3,2)$}}
\end{figure}
Hook numbers of partitions have very important implications for the representation theory of symmetric groups; the irreducible representations of $S_n$ are in bijection with partitions of $n$, and the dimensions of these representations are a function of the hook numbers of $\lambda$. Hook numbers have in recent decades also been the subject of many interesting arithmetic studies. This has arisen in large part due to the Nekrasov--Okounkov hook length formula, which connects hook numbers in a deep way to modular forms and $q$-series. This formula \cite{H,NO} says that for any complex number $z$, we have
\begin{align*}
    \sum_{\lambda \in \mathcal P} x^{|\lambda|}\prod_{h \in \mathcal H(\lambda)} \left(1-\frac{z}{h^2}\right) = \prod_{k=1}^\infty (1-x^k)^{z-1}.
\end{align*}
These connections have led to many interesting studies in the number theory and combinatorics literature regarding the asymptotic, combinatorial, and arithmetic properties of the hook numbers of partitions, especially studies into $t$-core partitions and $t$-hooks of partitions \cite{AS, BOW, CKNS, GKS, GO, H, OS}. 

The object of this paper is to prove a recent conjecture of the first author and collaborators \cite{BBCFW} on hook numbers in {\it restricted} classes of partitions, where analogues of the Nekrasov--Okounkov formula are not known. To motivate the question, recall Euler's famous theorem that the number of partitions of $n$ with only odd parts is exactly equal to the number of partitions of $n$ all of whose parts are distinct integers. It is natural to wonder whether various partition statistics behave differently on these families of partitions. For instance, it is easy to show that on average, partitions into odd parts have more parts than do partitions into distinct parts. One can see this, for example, by using Glaisher's bijection \cite{Glaisher}, which for an odd partition containing $m_d$ repetitions of the odd numbers $d$, creates at most $\lfloor \log_2(m_d) \rfloor + 1$ new distinct parts in a corresponding partition into distinct parts.
Another result from a paper of Andrews \cite{A} can be interpreted in terms of hook numbers\footnote{Andrews states his results in terms of distinct part sizes, which are easily seen to be equivalent to hook numbers equal to 1.}, which we now explain. Letting $\mathcal O(n)$ and $\mathcal D(n)$ be the sets of partitions of $n$ into odd parts and distinct parts, respectively, define, for any integer $h \geq 1$,
\begin{align*}
    a_h(n) := \sum_{\lambda \in \mathcal O(n)} \#\{ x \in \mathcal H(\lambda) : x = h \},\quad b_h(n) := \sum_{\lambda \in \mathcal D(n)} \#\{ x \in \mathcal H(\lambda) : x = h \}.
\end{align*}
Because of Euler's result that $\mathcal O(n)$ and $\mathcal D(n)$ are in bijection, we can see that for each $n \geq 0$, we have
\begin{align} \label{Balanced Equation}
    n p_{\mathcal O}(n) = \sum_{h \geq 1} a_h(n) = \sum_{h \geq 1} b_h(n) = n p_{\mathcal D}(n),
\end{align}
where we let $p_{\mathcal O}(n) = \left| \mathcal O(n) \right|$ and $p_{\mathcal D}(n) = \left| \mathcal D(n) \right|$. Andrews proved \cite{A} that for $n \geq 0$, we have $b_1(n) \geq a_1(n)$. 
In light of \eqref{Balanced Equation}, it would be natural to suspect that the inequality of Andrews must be balanced out in some way by inequalities between $a_h(n)$ and $b_h(n)$ for some $h>1$ which are in the other direction. In \cite{BBCFW}, this concept was formulated as a much more concrete conjecture:
\begin{conjecture} \label{BBCFW Conjecture}
    Let $h \geq 2$ be fixed. Then the following are true:
    \begin{itemize}
        \item[(1)] There exists some integer $N_h > 0$ such that for all $n > N_h$, we have $a_h(n) \geq b_h(n)$.
        \item[(2)] There exists some constant $\gamma_h > 1$ such that $a_h(n)/b_h(n) \to \gamma_h$ as $n \to \infty$.
    \end{itemize}
\end{conjecture}
Note that Conjecture \ref{BBCFW Conjecture} (2) is a much stronger statement than Conjecture \ref{BBCFW Conjecture}~(1). In \cite{BBCFW}, Conjecture \ref{BBCFW Conjecture} is proved only in the cases $h=2$ and $h=3$, but no progress is made for any cases $h \geq 4$. We also note that the authors of \cite{BBCFW} prove that $\gamma_1 = \frac{1}{2 \log(2)}, \gamma_2 = \frac 32$, and $\gamma_3 = \frac{2}{3\lp \log(2) - 1/8 \rp}$. In this paper, we improve on the methods of \cite{BBCFW}, which enable us to prove Conjecture \ref{BBCFW Conjecture} in its entirety.

\begin{theorem} \label{Main Theorem}
    Conjecture \ref{BBCFW Conjecture} (2), and therefore also Conjecture \ref{BBCFW Conjecture} (1), is true for all $h \geq 2$.
\end{theorem}

\begin{remark}
   In \cite{BBCFW}, the exact values of $N_2$ and $N_3$ are computed. In our case, although Theorem \ref{Main Theorem} establishes the existence of $N_h$, our proof is not effective. It would be possible to make the proof effective following the basic outline of \cite{C} if the main results of Section \ref{Asymptotics} could be made effective. Such an effective proof would not, however, give an optimal value for $N_h$. Based on computational data from \cite{BBCFW}, $N_h$ appears to grow approximately as $0.6 h^2$.
\end{remark}

In fact, the theorems we prove give much more detail about the behavior of $a_h(n)$ and $b_h(n)$, which we now summarize. In order to prove that $a_h(n)/b_h(n) \to \gamma_h$, we prove separate asymptotic formulas for $a_h(n)$ and $b_h(n)$, which we state below.

\begin{theorem} \label{Alpha Asymptotic Theorem}
    Let $h \geq 1$ be an integer. Then there exists a constant $\alpha_h \in \Q$ such that \begin{align*}
        a_h(n) \sim \alpha_h\dfrac{3^{1/4}}{2\pi n^{1/4}} e^{\pi \sqrt{n/3}}
    \end{align*}
    as $n \to \infty$.
\end{theorem}

\begin{theorem} \label{Beta Asymptotic Theorem}
    Let $h \geq 1$ be an integer. Then there exists a constant $\beta_h \in \Q\lp \log(2) \rp$ such that \begin{align*}
        b_h(n) \sim \beta_h \dfrac{3^{1/4}}{2\pi n^{1/4}} e^{\pi \sqrt{n/3}}
    \end{align*}
    as $n \to \infty$. Furthermore, $\beta_h \in \Q$ if and only if $h$ is even.
\end{theorem}

A key ingredient to these asymptotic formulas is the construction of the generating functions for $a_h(n)$ and $b_h(n)$, which previously had not been known apart from the cases $h=2$ and $h=3$. We make use of two representations of these generating functions, found in Theorems \ref{th:gf:ab} and \ref{th:rat:ab}. In particular, we show in Theorem \ref{th:rat:ab} that these generating functions are essentially rational functions of $q$ multiplied by the generating function for partitions into odd parts, which is a modular form.

From Theorems \ref{Alpha Asymptotic Theorem} and \ref{Beta Asymptotic Theorem}, it is immediately clear that $a_h(n)/b_h(n) \to \alpha_h/\beta_h$ as $h \to \infty$, so this establishes the existence of $\gamma_h$. By careful evaluations of the constants $\alpha_h$ and $\beta_h$, we are able to prove the following result for $\gamma_h$.

\begin{theorem} \label{Inequality Theorem}
    For each $h \geq 2$, we have $\gamma_h > 1$. Furthermore, we have $$\lim_{h\to\infty}\gamma_h=\dfrac{\log(4)}{\log(3)}.$$
\end{theorem}
\noindent Note that $\log(4)/\log(3)\approx1.2618...$.

It is now clear that to prove Theorem \ref{Main Theorem}, it will suffice to prove Theorems \ref{Alpha Asymptotic Theorem}, \ref{Beta Asymptotic Theorem}, and \ref{Inequality Theorem}. Therefore, the rest of the paper is dedicated to proving these central results. In Section \ref{Generating Functions}, we construct the generating functions for the sequences $a_h(n)$ and $b_h(n)$. We recall some known techniques in Section \ref{Preliminaries} which we will use to prove our theorems. In Section \ref{Asymptotics}, we prove Theorems \ref{Alpha Asymptotic Theorem} and \ref{Beta Asymptotic Theorem} using the circle method. In Section \ref{Constants} we then closely study the constants $\alpha_h$ and $\beta_h$ in order to prove Theorem \ref{Inequality Theorem}. We close with further conjectures and interesting probabilistic corollaries of our results in Section \ref{Discussion}; for example, we will show that for $n \gg 0$, most rows in partitions into distinct parts contain a hook of any given length $h \geq 1$.

\section*{Acknowledgements}

The authors wish to thank the anonymous referees for their helpful suggestions. We are also grateful for Koustav Banerjee's contribution of the simpler proof mentioned in Remark \ref{remark5}. The first author has received funding from the European Research Council (ERC) under the European Union’s Horizon 2020 research and innovation programme (grant agreement No. 101001179). The views expressed in this article are those of the author and do not reflect the official policy or position of the U.S. Naval Academy, Department of the Navy, the Department of Defense, or the U.S. Government.

\section{Generating Functions} \label{Generating Functions}

Recall the usual notation of the $q$-ascending
factorial 
\begin{align*}
(x;q)_n:=\begin{cases}
1,&\mbox{if }n=0;\\
\displaystyle\prod_{j=0}^{n-1}\left(1-xq^j\right),&\mbox{if }n\in\mathbb{N};
\end{cases}\qquad
(x;q)_\infty:=\lim_{n\rightarrow \infty} (x;q)_n.
\end{align*}
For $0\le k\le n$, 
let 
$$\displaystyle \binom{n}{k}_q:={(q;q)_n\over (q;q)_k\,(q;q)_{n-k}}$$ 
be the usual
$q$-binomial coefficient.
In this section, we establish the following explicit generating functions
for $a_h(n)$ and $b_h(n)$ by using a method described in \cite{BH}.
Let
\begin{equation*}
{\tilde{a}}_h(q) = \sum_{n=0}^\infty a_h(n)q^n
\text{\quad and \quad}
{\tilde{b}}_h(q) = \sum_{n=0}^\infty b_h(n)q^n.
\end{equation*}

\begin{theorem}\label{th:gf:ab}
For each $h\geq 1$, we have
\begin{align*}
{\tilde{a}}_h(q) &=(-q;q)_\infty
\sum_{j=0}^{\lceil h/2\rceil -1}
q^h
\binom{ h-j-1}{j}_{q^2} 
	\sum_{m\geq 0} (q^{2m+1};q^2)_j q^{(2h-4j)m}\\
& \qquad + (-q;q)_\infty
\sum_{j=0}^{\lfloor h/2\rfloor-1}
	q^{3h-4j-3 }
\binom{ h-j-2}{j}_{q^2} 
	\sum_{m\geq 0}
	 (q^{2m+3}; q^2)_j
	q^{(2h -4j   - 2)m   };\\
{\tilde{b}}_h(q) &=(-q;q)_\infty
\sum_{j=0}^{\lceil h/2\rceil -1}
  q^{h + j(j-1)/2} \binom{ h-j-1}{j}_q \sum_{m\geq 0}
{ q^{(j+1)m}\over (-q^{m+1};q)_{h-j}}.
\end{align*}

\end{theorem}

\begin{proof}

Each partition $\lambda$ can be represented by its Ferrers diagram.
For each box $v$ in the Ferrers diagram  of a partition $\lambda$, or
for each box $v$ in $\lambda$, for short, define the
{\it arm length} (resp. {\it leg length}, {\it coarm length},
{\it coleg length})
of~$v$, denoted by $\arm(\lambda, v)$
	(resp. $\leg(\lambda, v)$, $\coarm(\lambda, v)$, $\coleg(\lambda, v)$),
to be the number of
boxes $u$ such that $u$ lies
in the same row as $v$ and to the right of $v$
(resp. in the same column as $v$ and below $v$,
in the same row as $v$ and to the left of $v$,
in the same column as $v$ and above~$v$). See Figure 2.
\begin{figure}[H]
\begin{tikzpicture}[scale=0.800000]
\fill [white](7.8000,0.6000)--(7.8000,0.0000)--(7.2000,0.0000)--(6.6000,0.0000)--(6.6000,-0.6000)--(6.6000,-1.2000)--(6.0000,-1.2000)--(6.0000,-1.8000)--(5.4000,-1.8000)--(4.8000,-1.8000)--(4.8000,-2.4000)--(4.8000,-3.0000)--(4.2000,-3.0000)--(3.6000,-3.0000)--(3.6000,-3.6000)--(3.0000,-3.6000)--(3.0000,-4.2000)--(2.4000,-4.2000)--(1.8000,-4.2000)--(1.8000,-4.8000)--(1.2000,-4.8000)--(1.2000,-5.4000)--(1.2000,-6.0000)--(0.6000,-6.0000)--(0.6000,-6.6000)--(0.0000,-6.6000)--(0.0000,-6.0000)--(0.0000,-5.4000)--(0.0000,-4.8000)--(0.0000,-4.2000)--(0.0000,-3.6000)--(0.0000,-3.0000)--(0.0000,-2.4000)--(0.0000,-1.8000)--(0.0000,-1.2000)--(0.0000,-0.6000)--(0.0000,0.0000)--(0.0000,0.6000)--(0.6000,0.6000)--(1.2000,0.6000)--(1.8000,0.6000)--(2.4000,0.6000)--(3.0000,0.6000)--(3.6000,0.6000)--(4.2000,0.6000)--(4.8000,0.6000)--(5.4000,0.6000)--(6.0000,0.6000)--(6.6000,0.6000)--(7.2000,0.6000)--(7.8000,0.6000);
\draw [gray!10](0.0000,-6.6000)--(0.0000,0.6000);
\draw [gray!10](0.6000,-6.6000)--(0.6000,0.6000);
\draw [gray!10](1.2000,-6.0000)--(1.2000,0.6000);
\draw [gray!10](1.8000,-4.8000)--(1.8000,0.6000);
\draw [gray!10](2.4000,-4.2000)--(2.4000,0.6000);
\draw [gray!10](3.0000,-4.2000)--(3.0000,0.6000);
\draw [gray!10](3.6000,-3.6000)--(3.6000,0.6000);
\draw [gray!10](4.2000,-3.0000)--(4.2000,0.6000);
\draw [gray!10](4.8000,-3.0000)--(4.8000,0.6000);
\draw [gray!10](5.4000,-1.8000)--(5.4000,0.6000);
\draw [gray!10](6.0000,-1.8000)--(6.0000,0.6000);
\draw [gray!10](6.6000,-1.2000)--(6.6000,0.6000);
\draw [gray!10](7.2000,0.0000)--(7.2000,0.6000);
\draw [gray!10](0.0000,-6.6000)--(0.6000,-6.6000);
\draw [gray!10](0.0000,-6.0000)--(1.2000,-6.0000);
\draw [gray!10](0.0000,-5.4000)--(1.2000,-5.4000);
\draw [gray!10](0.0000,-4.8000)--(1.8000,-4.8000);
\draw [gray!10](0.0000,-4.2000)--(3.0000,-4.2000);
\draw [gray!10](0.0000,-3.6000)--(3.6000,-3.6000);
\draw [gray!10](0.0000,-3.0000)--(4.8000,-3.0000);
\draw [gray!10](0.0000,-2.4000)--(4.8000,-2.4000);
\draw [gray!10](0.0000,-1.8000)--(6.0000,-1.8000);
\draw [gray!10](0.0000,-1.2000)--(6.6000,-1.2000);
\draw [gray!10](0.0000,-0.6000)--(6.6000,-0.6000);
\draw [gray!10](0.0000,0.0000)--(7.8000,0.0000);
\draw [black](7.8000,0.6000)--(7.8000,0.0000)--(7.2000,0.0000)--(6.6000,0.0000)--(6.6000,-0.6000)--(6.6000,-1.2000)--(6.0000,-1.2000)--(6.0000,-1.8000)--(5.4000,-1.8000)--(4.8000,-1.8000)--(4.8000,-2.4000)--(4.8000,-3.0000)--(4.2000,-3.0000)--(3.6000,-3.0000)--(3.6000,-3.6000)--(3.0000,-3.6000)--(3.0000,-4.2000)--(2.4000,-4.2000)--(1.8000,-4.2000)--(1.8000,-4.8000)--(1.2000,-4.8000)--(1.2000,-5.4000)--(1.2000,-6.0000)--(0.6000,-6.0000)--(0.6000,-6.6000)--(0.0000,-6.6000)--(0.0000,-6.0000)--(0.0000,-5.4000)--(0.0000,-4.8000)--(0.0000,-4.2000)--(0.0000,-3.6000)--(0.0000,-3.0000)--(0.0000,-2.4000)--(0.0000,-1.8000)--(0.0000,-1.2000)--(0.0000,-0.6000)--(0.0000,0.0000)--(0.0000,0.6000)--(0.6000,0.6000)--(1.2000,0.6000)--(1.8000,0.6000)--(2.4000,0.6000)--(3.0000,0.6000)--(3.6000,0.6000)--(4.2000,0.6000)--(4.8000,0.6000)--(5.4000,0.6000)--(6.0000,0.6000)--(6.6000,0.6000)--(7.2000,0.6000)--(7.8000,0.6000);
\fill [gray!40, line width=0.4pt](0.0000,-2.4000)--(1.8000,-2.4000)--(1.8000,-1.8000)--(0.0000,-1.8000)--(0.0000,-2.4000);
\draw [gray!200, line width=0.4pt](0.0000,-2.4000)--(1.8000,-2.4000)--(1.8000,-1.8000)--(0.0000,-1.8000)--(0.0000,-2.4000);
\fill [gray!40, line width=0.4pt](2.4000,-2.4000)--(4.8000,-2.4000)--(4.8000,-1.8000)--(2.4000,-1.8000)--(2.4000,-2.4000);
\draw [gray!200, line width=0.4pt](2.4000,-2.4000)--(4.8000,-2.4000)--(4.8000,-1.8000)--(2.4000,-1.8000)--(2.4000,-2.4000);
\fill [gray!40, line width=0.4pt](1.8000,-4.2000)--(2.4000,-4.2000)--(2.4000,-2.4000)--(1.8000,-2.4000)--(1.8000,-4.2000);
\draw [gray!200, line width=0.4pt](1.8000,-4.2000)--(2.4000,-4.2000)--(2.4000,-2.4000)--(1.8000,-2.4000)--(1.8000,-4.2000);
\fill [gray!40, line width=0.4pt](1.8000,-1.8000)--(2.4000,-1.8000)--(2.4000,0.6000)--(1.8000,0.6000)--(1.8000,-1.8000);
\draw [gray!200, line width=0.4pt](1.8000,-1.8000)--(2.4000,-1.8000)--(2.4000,0.6000)--(1.8000,0.6000)--(1.8000,-1.8000);
\draw (2.1000, -2.1000) node [] {$v$};
\draw [black](1.8000,-2.4000)--(1.8000,-1.8000)--(2.4000,-1.8000)--(2.4000,-2.4000)--(1.8000,-2.4000);
\draw (3.6000, -2.1000) node [] {$j$};
\draw (1.0200, -2.1000) node [] {$m$};
\draw (2.1000, -0.6000) node [] {$g$};
\draw (2.1000, -3.3000) node [] {$l$};
\end{tikzpicture}
\caption{{\small Arm, leg, coarm, and coleg lengths: $\arm(\lambda, v)=j$, $\leg(\lambda, v)=l$, $\coarm(\lambda, v)=m$, $\coleg(\lambda, v)=g$}}
\end{figure}

Consider a set $\mathcal L$ of partitions.
For each given triplet $(j,l,m)$  of integers,
	let $f_{\mathcal L}(j,l,m;n)$ denote
the number of ordered pairs $(\lambda,v)$ such that
	$\lambda \in \mathcal L$, $v\in\lambda$, $\lambda \vdash n$, $\arm(\lambda, v)=j$, $\leg(\lambda, v)=l$, 
	and $\coarm(\lambda, v)=m$.
For a fixed partition $\lambda$, it is easy to see that for each box $v \in \lambda$ and $\lambda \vdash n$, the triplets
	$(\arm(\lambda, v), \leg(\lambda, v),\coarm(\lambda, v))$ are different. Now, let the triplet
$(j,l,m)$ be fixed. The generating function for those partitions is equal to the product of
several ``small" generating functions for the different regions of the partitions,
as shown in Figure 3.
\begin{figure}[H]
\begin{tikzpicture}[scale=0.800000]
\fill [white](7.8000,0.6000)--(7.8000,0.0000)--(7.2000,0.0000)--(6.6000,0.0000)--(6.6000,-0.6000)--(6.6000,-1.2000)--(6.0000,-1.2000)--(6.0000,-1.8000)--(5.4000,-1.8000)--(4.8000,-1.8000)--(4.8000,-2.4000)--(4.8000,-3.0000)--(4.2000,-3.0000)--(3.6000,-3.0000)--(3.6000,-3.6000)--(3.0000,-3.6000)--(3.0000,-4.2000)--(2.4000,-4.2000)--(1.8000,-4.2000)--(1.8000,-4.8000)--(1.2000,-4.8000)--(1.2000,-5.4000)--(1.2000,-6.0000)--(0.6000,-6.0000)--(0.6000,-6.6000)--(0.0000,-6.6000)--(0.0000,-6.0000)--(0.0000,-5.4000)--(0.0000,-4.8000)--(0.0000,-4.2000)--(0.0000,-3.6000)--(0.0000,-3.0000)--(0.0000,-2.4000)--(0.0000,-1.8000)--(0.0000,-1.2000)--(0.0000,-0.6000)--(0.0000,0.0000)--(0.0000,0.6000)--(0.6000,0.6000)--(1.2000,0.6000)--(1.8000,0.6000)--(2.4000,0.6000)--(3.0000,0.6000)--(3.6000,0.6000)--(4.2000,0.6000)--(4.8000,0.6000)--(5.4000,0.6000)--(6.0000,0.6000)--(6.6000,0.6000)--(7.2000,0.6000)--(7.8000,0.6000);
\draw [gray!10](0.0000,-6.6000)--(0.0000,0.6000);
\draw [gray!10](0.6000,-6.6000)--(0.6000,0.6000);
\draw [gray!10](1.2000,-6.0000)--(1.2000,0.6000);
\draw [gray!10](1.8000,-4.8000)--(1.8000,0.6000);
\draw [gray!10](2.4000,-4.2000)--(2.4000,0.6000);
\draw [gray!10](3.0000,-4.2000)--(3.0000,0.6000);
\draw [gray!10](3.6000,-3.6000)--(3.6000,0.6000);
\draw [gray!10](4.2000,-3.0000)--(4.2000,0.6000);
\draw [gray!10](4.8000,-3.0000)--(4.8000,0.6000);
\draw [gray!10](5.4000,-1.8000)--(5.4000,0.6000);
\draw [gray!10](6.0000,-1.8000)--(6.0000,0.6000);
\draw [gray!10](6.6000,-1.2000)--(6.6000,0.6000);
\draw [gray!10](7.2000,0.0000)--(7.2000,0.6000);
\draw [gray!10](0.0000,-6.6000)--(0.6000,-6.6000);
\draw [gray!10](0.0000,-6.0000)--(1.2000,-6.0000);
\draw [gray!10](0.0000,-5.4000)--(1.2000,-5.4000);
\draw [gray!10](0.0000,-4.8000)--(1.8000,-4.8000);
\draw [gray!10](0.0000,-4.2000)--(3.0000,-4.2000);
\draw [gray!10](0.0000,-3.6000)--(3.6000,-3.6000);
\draw [gray!10](0.0000,-3.0000)--(4.8000,-3.0000);
\draw [gray!10](0.0000,-2.4000)--(4.8000,-2.4000);
\draw [gray!10](0.0000,-1.8000)--(6.0000,-1.8000);
\draw [gray!10](0.0000,-1.2000)--(6.6000,-1.2000);
\draw [gray!10](0.0000,-0.6000)--(6.6000,-0.6000);
\draw [gray!10](0.0000,0.0000)--(7.8000,0.0000);
\draw [black](7.8000,0.6000)--(7.8000,0.0000)--(7.2000,0.0000)--(6.6000,0.0000)--(6.6000,-0.6000)--(6.6000,-1.2000)--(6.0000,-1.2000)--(6.0000,-1.8000)--(5.4000,-1.8000)--(4.8000,-1.8000)--(4.8000,-2.4000)--(4.8000,-3.0000)--(4.2000,-3.0000)--(3.6000,-3.0000)--(3.6000,-3.6000)--(3.0000,-3.6000)--(3.0000,-4.2000)--(2.4000,-4.2000)--(1.8000,-4.2000)--(1.8000,-4.8000)--(1.2000,-4.8000)--(1.2000,-5.4000)--(1.2000,-6.0000)--(0.6000,-6.0000)--(0.6000,-6.6000)--(0.0000,-6.6000)--(0.0000,-6.0000)--(0.0000,-5.4000)--(0.0000,-4.8000)--(0.0000,-4.2000)--(0.0000,-3.6000)--(0.0000,-3.0000)--(0.0000,-2.4000)--(0.0000,-1.8000)--(0.0000,-1.2000)--(0.0000,-0.6000)--(0.0000,0.0000)--(0.0000,0.6000)--(0.6000,0.6000)--(1.2000,0.6000)--(1.8000,0.6000)--(2.4000,0.6000)--(3.0000,0.6000)--(3.6000,0.6000)--(4.2000,0.6000)--(4.8000,0.6000)--(5.4000,0.6000)--(6.0000,0.6000)--(6.6000,0.6000)--(7.2000,0.6000)--(7.8000,0.6000);
\fill [gray!40, line width=0.4pt](0.0000,-2.4000)--(1.8000,-2.4000)--(1.8000,-1.8000)--(0.0000,-1.8000)--(0.0000,-2.4000);
\draw [gray!200, line width=0.4pt](0.0000,-2.4000)--(1.8000,-2.4000)--(1.8000,-1.8000)--(0.0000,-1.8000)--(0.0000,-2.4000);
\fill [gray!40, line width=0.4pt](2.4000,-2.4000)--(4.8000,-2.4000)--(4.8000,-1.8000)--(2.4000,-1.8000)--(2.4000,-2.4000);
\draw [gray!200, line width=0.4pt](2.4000,-2.4000)--(4.8000,-2.4000)--(4.8000,-1.8000)--(2.4000,-1.8000)--(2.4000,-2.4000);
\fill [gray!40, line width=0.4pt](1.8000,-4.2000)--(2.4000,-4.2000)--(2.4000,-2.4000)--(1.8000,-2.4000)--(1.8000,-4.2000);
\draw [gray!200, line width=0.4pt](1.8000,-4.2000)--(2.4000,-4.2000)--(2.4000,-2.4000)--(1.8000,-2.4000)--(1.8000,-4.2000);
\fill [gray!40, line width=0.4pt](0.0000,-4.2000)--(1.8000,-4.2000)--(1.8000,-4.1760)--(0.0000,-4.1760)--(0.0000,-4.2000);
\draw [gray!200, line width=0.4pt](0.0000,-4.2000)--(1.8000,-4.2000)--(1.8000,-4.1760)--(0.0000,-4.1760)--(0.0000,-4.2000);
\draw (2.1000, -2.1000) node [] {$v$};
\draw [black](1.8000,-2.4000)--(1.8000,-1.8000)--(2.4000,-1.8000)--(2.4000,-2.4000)--(1.8000,-2.4000);
\draw (3.6000, -2.1000) node [] {$j$};
\draw (1.0200, -2.1000) node [] {$m$};
\draw (2.1000, -3.3000) node [] {$l$};
\draw (0.9000, -3.3000) node [] {$D$};
\draw (3.0000, -3.0000) node [] {$B$};
\draw (0.6000, -5.1000) node [] {$A$};
\draw (3.3000, -0.6000) node [] {$C$};
\end{tikzpicture}
\caption{{\small A partition and its different regions}}
\end{figure}

We first examine $\tilde b_h(q)$. Let $F_{\mathcal D}(j,l,m; q):=\sum_n f_{\mathcal D}(j,l,m;n) q^n$ be the generating function for partitions into distinct parts with arm length $j$, leg length $l$, and coarm length $m$, as displayed in the above diagram. It is quite routine (see, e.g., \cite[Chap.~3]{Andrews}) to prove that
$$
	F_{\mathcal D}(j,l,m; q)
	= 
	F_{\mathcal D}(A, q)F_{\mathcal D}(B, q) 
	F_{\mathcal D}(C, q)F_{\mathcal D}(D, q),
$$
where
 the generating functions for the regions $A,B,C,D$ are respectively
\begin{align*}
F_{\mathcal D}(A, q)&=(-q;q)_m,\\
F_{\mathcal D}(B, q)&=\binom{ j}{l}_q q^{l(l-1)/2},\\
F_{\mathcal D}(C, q)&=\left(-q^{m+j+2};q\right)_\infty= { (-q;q)_\infty \over (-q;q)_{m+j+1}},\\
F_{\mathcal D}(D, q)&=q^{(m+1)(l+1)+j}.
\end{align*}
Hence,
\begin{align*}
F_{\mathcal D}(j,l,m; q)&=F_{\mathcal D}(A, q)F_{\mathcal D}(B, q) 
	F_{\mathcal D}(C, q)F_{\mathcal D}(D, q)\\
&=(-q;q)_m
\binom{j}{l}_q q^{l(l-1)/2}
{ (-q;q)_\infty\over (-q;q)_{m+j+1}}
q^{(m+1)(l+1)+j}\\
&=q^{(m+1)(l+1)+j + l(l-1)/2}
\binom{ j}{l}_q 
{ (-q;q)_\infty\over (-q^{m+1};q)_{j+1}}.
\end{align*}
As $\binom{j}{l}_q = 0$ for $j<l$, we see that $F_{\mathcal D}(j,l,m; q)=0$. 
Since $h=j+l+1$, the condition $j\geq l$ implies that
$ h-l-1 \geq l $, and thus we have $l\leq (h-1)/2 $.
We therefore have
\begin{align*}
\tilde b_h(q)
&=\sum_{l=0}^{\lceil h/2 \rceil -1} \sum_{m\geq 0} F_{\mathcal D}(h-l-1,l,m;q)\\
&=\sum_{l=0}^{\lceil h/2 \rceil -1} \sum_{m\geq 0} 
q^{(m+1)(l+1)+(h-l-1) + l(l-1)/2}
\binom{ h-l-1}{l}_q 
{ (-q;q)_\infty\over (-q^{m+1};q)_{h-l-1+1}},
\end{align*}
which is equal to the expression given in the theorem.

\medskip

Next, we examine $\tilde{a}_h(q)$. Let $F_{\mathcal O}(j,l,m; q):=\sum_n f_{\mathcal O}(j,l,m;n) q^n$ be the generating function for partitions into odd parts with arm length $j$, leg length $l$, and coarm length $m$.  If $m=2m'$ is even, then $j=2j'$ is also even. 
With $j+l+1=h$, we have $2j'+l+1=h$ and $ j'\leq (h-1)/2 $.
The generating functions for the regions $A,B,C,D$ are respectively
\begin{align*}
F_{\mathcal O}^e(A, q)&=\frac{1}{(q;q^2)_{m'}},\quad F_{\mathcal O}^e(B, q)=\binom{ j'+l}{j'}_{q^2} ,\\
F_{\mathcal O}^e(C, q)&=\frac{1}{(q^{2m'+2j'+1}; q^2)_\infty},\quad F_{\mathcal O}^e(D, q)=q^{(2m'+1)(l+1)+2j'},
\end{align*}
where to obtain $F_{\mathcal O}^e(B, q)$, we made the following calculation:
\begin{equation*}
\sum_{l\geq 0} F_{\mathcal O}^e(B, q)t^l
=\frac{1}{(t;q^2)_{j'+1}}
=\sum_{l\geq 0} \binom{ j'+l}{l}_{q^2} t^l.
\end{equation*}
Hence, with $2j'+l+1=h$, we have
\begin{align*}
F_{\mathcal O}^e(2j',l, 2m';q)&=F_{\mathcal O}^e(A, q)F_{\mathcal O}^e(B, q) 
	F_{\mathcal O}^e(C, q)F_{\mathcal O}^e(D, q)\\
&=\frac{1}{(q;q^2)_{m'}}
\binom{ j'+l}{j'}_{q^2} 
{1\over (q^{2m'+2j'+1};q^2)_\infty }
q^{(2m'+1)(l+1)+2j'}\\
&=\frac{1}{(q;q^2)_{\infty}}
\binom{ h-j'-1}{j'}_{q^2} 
{(q^{2m'+1};q^2)_{j'}}
q^{h + (2h-4j')m' }.
\end{align*}

\medskip

If $m=2m'+1$ is odd, then $j=2j'+1$ is also odd.
Since $j+l+1=h$, we have $2j'+1+l+1=h$ and $j'\leq (h-2)/2 $.
The generating functions for the regions $A,B,C,D$ are respectively
\begin{align*}
F_{\mathcal O}^o(A, q)&=\frac{1}{(q;q^2)_{m'+1}},\quad F_{\mathcal O}^o(B, q)= \binom{j'+l}{j'}_{q^2} q^l,\\
F_{\mathcal O}^o(C, q)&=\frac{1}{(q^{2m'+2j'+3}; q^2)_\infty},\quad F_{\mathcal O}^o(D, q)=q^{(2m'+2)(l+1)+2j'+1},
\end{align*}
where
where to obtain $F_{\mathcal O}^o(B, q)$, we made the following calculation:
\begin{equation*}
\sum_{l\geq 0} F_{\mathcal O}^o(B, q)t^l
=\frac{1}{(tq;q^2)_{j'+1}}
=\sum_{l\geq 0} \binom{ j'+l}{l}_{q^2} (tq)^l.
\end{equation*}
Hence, for $2j'+1+l+1=h$, we have
\begin{align*}
F_{\mathcal O}^o(2j'+1,l,2m'+1;q)&=F_{\mathcal O}^o(A, q)F_{\mathcal O}^o(B, q) 
	F_{\mathcal O}^o(C, q)F_{\mathcal O}^o(D, q)\\
&=\frac{1}{(q;q^2)_{m'+1}}
\binom{ j'+l}{j'}_{q^2}  q^l
{1\over (q^{2m'+2j'+3};q^2)_\infty }
q^{(2m'+2)(l+1)+2j'+1}\\
&=\frac{1}{(q;q^2)_{\infty}}
\binom{ h-j'-2}{j'}_{q^2}  
 (q^{2m'+3};q^2)_{j'}
q^{3h-4j'-3+(2h-4j'-2)m'}.
\end{align*}
Using Euler's well-known identity $(-q;q)_\infty (q;q^2)_\infty=1$, combining the above 
two cases yields 
\begin{align*}
\tilde{a}_h(q)
&= \sum_{j'=0}^{\lceil h/2 \rceil -1} \sum_{m'\geq 0} F_{\mathcal O}^e(2j',h-2j'-1, 2m';q)\\
&\qquad+ \sum_{j'=0}^{\lfloor h/2 \rfloor -1} \sum_{m'\geq 0} F_{\mathcal O}^o(2j'+1,h-2j'-2,2m'+1;q)\\
&= \sum_{j'=0}^{\lceil h/2 \rceil -1} \sum_{m'\geq 0} 
\frac{1}{(q;q^2)_{\infty}}
\binom{ h-j'-1}{j'}_{q^2} 
{(q^{2m'+1};q^2)_{j'}}
q^{h + (2h-4j')m' }\\
&\qquad+ \sum_{j'=0}^{\lfloor h/2 \rfloor -1} \sum_{m'\geq 0} 
\frac{1}{(q;q^2)_{\infty}}
\binom{ h-j'-2}{j'}_{q^2}  
 (q^{2m'+3};q^2)_{j'}
q^{3h-4j'-3+(2h-4j'-2)m'}\\
&= (-q;q)_\infty\sum_{j=0}^{\lceil h/2 \rceil -1} \sum_{m\geq 0} 
q^h\binom{ h-j-1}{j}_{q^2} 
{(q^{2m+1};q^2)_{j}}
q^{(2h-4j)m }\\
&\qquad+ (-q;q)_\infty\sum_{j=0}^{\lfloor h/2 \rfloor -1} \sum_{m\geq 0} 
q^{3h-4j-3} \binom{ h-j-2}{j}_{q^2}  
 (q^{2m+3};q^2)_{j}
q^{(2h-4j-2)m}.
\end{align*}
which is exactly the expression for $\tilde{a}_h(q)$  given in the theorem.
\end{proof}

From Theorem \ref{th:gf:ab}, we can derive the following result, which shows the relatively elementary nature of the generating functions for $a_h(n)$ and $b_h(n)$.

\begin{theorem}\label{th:rat:ab}
We have
\begin{align*}
{\tilde{a}}_{h}(q) &=(-q;q)_\infty\ {\tilde{A}}_{h}(q),\\
{\tilde{b}}_{2h}(q) &=(-q;q)_\infty\ {\tilde{B}}_{2h}(q),\\
{\tilde{b}}_{2h+1}(q) &=(-q;q)_\infty\ \Bigl( {\tilde{B}}_{2h+1}(q) + 
\sum_{k=1}^\infty \frac{q^{k}}{1+q^{k}}\Bigr),
\end{align*}
where ${\tilde{A}}_h(q)$ and ${\tilde{B}}_h(q)$ are  rational functions in $q$.
\end{theorem}
Here are the first few values of ${\tilde{A}}_h(q)$ and ${\tilde{B}}_h(q)$:
\begin{align*}
{\tilde{A}}_1(q)&=\frac{q}{1-q^2},
\qquad
{\tilde{A}}_2(q)=\frac{ q^{5} + q^{3} + q^{2}  }{1-q^4},\\
{\tilde{A}}_3(q)&=\frac{ q^{10} + q^{9} + 2 \, q^{7} + 3 \, q^{5} - q^{4} + 2 \, q^{3}  }{(1-q^6) (1+q^2)},\\
{\tilde{B}}_1(q)&=0,
\qquad
{\tilde{B}}_2(q)=\frac{q^{2}}{1-q^2},
\qquad
{\tilde{B}}_3(q)=\frac{q^{5} - q^{2} - q }{(1-q^4) (1+q)}.
\end{align*}
\begin{proof}
The basic idea of the proof is to 
simplify the expressions given in Theorem \ref{th:gf:ab} by removing $\sum_{m\geq 0}$ via the geometric summation formula.
We have the following useful identities, which are consequences of well-known identities for $q$-binomial coefficients (for example, see \cite[Ch. 3]{Andrews}):
\begin{align}
(q^{2m+1};q^2)_n &=\sum_{k=0}^n \binom{n}{k}_{q^2}(-1)^k  q^{2mk+k^2}, \label{eq 1}\\
\frac{1}{(-q^{m+1};q)_n} &=\sum_{k=0}^\infty \binom{n +k-1}{k}_q (-1)^k  q^{(m+1)k}, \label{eq 2}\\
\sum_{k=1}^\infty \frac{(-q)^k}{1-q^k}
&=- \sum_{k=1}^\infty \frac{q^k}{1+q^k}.\label{eq 3}
\end{align}

From Theorem \ref{th:gf:ab}, we have 
\begin{equation}\label{ah_for_asymptotic}
\frac{{\tilde{a}}_h(q)}{(-q;q)_\infty
} = 
	\sum_{j=0}^{\lceil h/2\rceil -1}
\binom{h-j-1}{j}_{q^2}  A_1(h,j;q)
	+ \sum_{j=0}^{\lfloor h/2\rfloor-1}
\binom{ h-j-2}{j}_{q^2}  A_2(h,j;q),
\end{equation}
where, with the help of \eqref{eq 1} we have
\begin{align}\label{A_1}
A_1(h,j;q)
&:=q^h
\sum_{m\geq 0} 
{(q^{2m+1}; q^2)_j  q^{ (2h-4j)m} }\\
&=q^h
\sum_{m\geq 0} 
 \sum_{k=0}^j \binom{j}{k}_{q^2}(-1)^k  q^{2mk+k^2}  q^{ (2h-4j)m}\nonumber\\
&=q^{h}
 \sum_{k=0}^j
 \binom{j}{k}_{q^2}(-1)^k q^{k^2}
\sum_{m\geq 0} 
 (q^{2h+2k-4j})^m\nonumber\\
&=q^{h}
 \sum_{k=0}^j
 \binom{j}{k}_{q^2}
 \frac{(-1)^k q^{k^2}}{1-q^{2h+2k-4j}},\nonumber
\end{align}
and
\begin{align}\label{A_2}
A_2(h,j;q)
&:=q^{3h-4j-3}
\sum_{m\geq 0} 
{(q^{2m+3}; q^2)_j q^{(2h-4j  - 2)m  }
}\\
&=q^{3h-4j-3}
\sum_{m\geq 0} 
 \sum_{k=0}^j \binom{j}{k}_{q^2}(-1)^k  q^{2(m+1)k+k^2} 
q^{(2h -4j  - 2)m  }\nonumber\\
&=q^{3h-4j-3}
\sum_{k=0}^j
 \binom{j}{k}_{q^2}(-1)^k q^{k^2+2k} 
\sum_{m\geq 0}  (q^{2h -4j  + 2k - 2  })^m\nonumber\\
&=q^{3h-4j-3}
\sum_{k=0}^j
 \binom{j}{k}_{q^2} 
 \frac{(-1)^k q^{k^2+2k}}{1-q^{2h -4j  + 2k - 2  }}.\nonumber
\end{align}
Thus, it is clear that ${\tilde{A}}_h(q)$ is a rational function in $q$.

We now consider $\tilde{b}_h(q)$. From Theorem \ref{th:gf:ab}, we write
\begin{equation}\label{bh_for_asymptotic}
{\tilde{b}}_h(q) =
(-q;q)_\infty
\sum_{j=0}^{\lceil h/2\rceil -1}
	q^{h + j(j-1)/2} \binom{ h-j-1}{j}_q Y(h,j;q),
\end{equation}
where, with the help of \eqref{eq 2},
\begin{align*}
Y(h,j;q)
&:=\sum_{m\geq 0} { q^{m(j+1)}\over (-q^{m+1};q)_{h-j}}\\
&=\sum_{m\geq 0} { q^{m(j+1)}}
\sum_{k=0}^\infty \binom{h-j +k-1}{k}_q (-1)^k  q^{(m+1)k}\\
&=\sum_{k=0}^\infty
(-q)^k \binom{h-j +k-1}{k}_q
\sum_{m\geq 0} { q^{m(j+k+1)}}\\
&=\sum_{k=0}^\infty
 \binom{h-j +k-1}{k}_q
\frac{(-q)^k}{1- q^{j+k+1}}.
\end{align*}
Since
\begin{equation*}
\binom{h-j +k-1}{k}_q
=
\frac{\left(q^{k+1};q\right)_{h-j-1}}{(q;q)_{h-j-1}},
\end{equation*}
there are two cases to be considered.

\underline{Case 1}: Suppose $k+h-j-1\geq j+k+1$; i.e., 
$j\leq h/2-1$.
Then the denominator term $1-q^{j+k+1}$ is cancelled with a factor in the numerator. Hence, 
$$ 
\binom{h-j +k-1}{k}_q
\frac{(-q)^k}{1- q^{j+k+1}} = \frac{1}{(q;q)_{h-j-1}} \sum_{s=1}^{h-j-1} C_s(h,j; q) q^{sk},
$$
where $C_s(h,j; q)$ are polynomials in $q$. Notice that $C_s(h,j;q)$ are independent of $k$. Thus,  
\begin{align*}
Y(h,j;q)
=\sum_{k=0}^\infty\frac{1}{(q;q)_{h-j-1}} \sum_{s=0}^{h-j-1} C_s(h,j; q) q^{sk}
=\frac{1}{(q;q)_{h-j-1}} \sum_{s=0}^{h-j-1}  \frac{C_s(h,j; q)}{1-q^{s}}.
\end{align*}
Hence, $Y(h,j;q)$ is a rational function in $q$, so that
$ {\tilde{b}}_h(q)/ (-q;q)_\infty  $ is also a rational function in $q$.

\underline{Case 2}: Suppose on the other hand that $k+h-j-1< j+k+1$; i.e., $j> h/2-1$.
Since $j\leq \lceil h/2\rceil-1$, 
the only possible choice is that $h=2j+1$ is odd. We then have
\begin{align*}
Y(2j+1,j;q)
&=\sum_{k=0}^\infty
 \binom{2j+1-j +k-1}{k}_q
\frac{(-q)^k}{1- q^{j+k+1}}\\
&=\sum_{k=0}^\infty
\frac{\left(q^{k+1};q\right)_j(-q)^k}{(q;q)_j\left(1- q^{j+k+1}\right)}\\
&=\frac{ 1}{(q;q)_{j}} \sum_{k= 0}^\infty Z_k(j,q^k;q),
\end{align*}
with
\begin{align*}
Z_k(j,x;q)
&:=(-1)^k
\frac{x(xq;q)_j}{1- xq^{j+1}}.
\end{align*}
We write
\begin{equation*}
Z_k(j,x;q)=(-1)^k \left(F(x,q)+\frac{G(q)}{1-xq^{j+1}}\right),
\end{equation*}
where $F(x,q),G(q)$, viewed as polynomials of $x$,  are determined by
$$
F(x,q)+\frac{G(q)}{1-xq^{j+1}}=\frac{x(xq;q)_j}{1- xq^{j+1}}.
$$
Taking $x\rightarrow 0$, we see that $G(q)=-F(0,q)$ and
\begin{equation*}
{G(q)}
={x(xq;q)_j} \Big|_{x=q^{-j-1}}
=\frac{(-1)^j}{q^{(j+1)(j+2)/2}}(q;q)_j.
\end{equation*}
On the other hand, with the help of \eqref{eq 3} we have
\begin{align*}
\sum_{k=0}^\infty \frac{(-1)^k q^{k+j+1}}{1-q^{k+j+1}}
&=(-1)^{j+1} \sum_{k=0}^\infty \frac{(-q)^{k+j+1}}{1-q^{k+j+1}}\\
&=(-1)^{j+1} \sum_{k=1}^\infty \frac{(-q)^{k}}{1-q^{k}}
+  (-1)^{j}\sum_{k=1}^j \frac{(-q)^{k}}{1-q^{k}}\\
&=(-1)^j \sum_{k=1}^\infty \frac{q^{k}}{1+q^{k}}
+ (-1)^j\sum_{k=1}^j \frac{(-q)^{k}}{1-q^{k}}.
\end{align*}
Therefore, we obtain
\begin{align*}
Y(2j+1,j;q)
&=\frac{ 1}{(q;q)_{j}} \sum_{k= 0}^\infty Z_k(j,q^k;q)\\
&=\frac{ 1}{(q;q)_{j}} \sum_{k= 0}^\infty 
(-1)^k \left(F(q^k,q)+G(q)+\frac{G(q)q^{k+j+1}}{1-q^{k+j+1}}\right)\\
&=\frac{ 1}{(q;q)_{j}} \sum_{k= 0}^\infty 
(-1)^k \left(F(q^k,q)-F(0, q)\right)\\
& \qquad +
\frac{ G(q)}{(q;q)_{j}} \sum_{k= 0}^\infty 
(-1)^k \left(\frac{q^{k+j+1}}{1-q^{k+j+1}}\right)\\
&=\text{Rational($q$)} +
\frac{ (-1)^j }{q^{(j+1)(j+2)/2}} \sum_{k= 0}^\infty 
(-1)^k \left(\frac{q^{k+j+1}}{1-q^{k+j+1}}\right)\\
&=\text{Rational($q$)} + 
\frac{ 1 }{q^{(j+1)(j+2)/2}} 
 \sum_{k=1}^\infty \frac{q^{k}}{1+q^{k}},
\end{align*}
where $\text{Rational}(q)$ is some rational function in $q$ which may vary from one step to the next.  In the fourth identity of the above equations, 
$$
\sum_{k= 0}^\infty 
(-1)^k \left(F(q^k,q)-F(0, q)\right)
$$
is a rational function in $q$, because $F(x,q)$ is a polynomial in $x$.

Finally, when $h=2m+1$ is odd, we therefore have
\begin{align*}
\frac{{\tilde{b}}_h(q)}{
	(-q;q)_\infty
} 
&=\sum_{j=0}^{m-1}
	q^{h + j(j-1)/2} \binom{h-j-1}{j}_q Y(h,j;q)+ q^{2m+1+m(m-1)/2} Y(2m+1,m;q)\\
&=\text{Rational($q$)} + \sum_{k=1}^\infty \frac{q^{k}}{1+q^{k}}.
\end{align*}
This completes the proof.
\end{proof}

\section{Preliminaries}\label{Preliminaries}

In order to prove Theorems \ref{Alpha Asymptotic Theorem} and \ref{Beta Asymptotic Theorem}, we must first recall Euler--Maclaurin summation and Wright's circle method, as well as some consequences of these methods.

\subsection{Euler--Maclaurin Summation}\label{Euler-Maclaurin Summation Section}

Euler--Maclaurin summation gives a method for estimating, and even obtaining exact formulas for, sums involving terms of the form $f(mz)$. There are many formulations, some asymptotic and some exact. We will use the following formulation, which comes from \cite{BJU}. We actually use Euler--Maclaurin summation in a slightly unusual manner to analyze sums whose terms do not quite fit the form $f(mz)$, inspired by the methodology applied in \cite{BCM}.

Suppose that $D_\theta := \{ r e^{i\alpha} : r \geq 0, |\alpha| \leq \theta \}$ for some $0 \leq \theta < \frac{\pi}{2}$. Let $f : \C \to \C$ be holomorphic in the domain $D_\theta$, and assume that $f$ and all of its derivatives decay at infinity at least as fast as $|z|^{-1-\varepsilon}$ for some $\varepsilon > 0$. (This property is often called {\it sufficient decay}.) Finally, let $B_n(x)$ denote the Bernoulli polynomials, and let $\widetilde B_n(x) := B_n\lp \{ x \} \rp$ be the periodic Bernoulli function.  Then from \cite[Equation (5.7)]{BJU} in the special case $a=1$, we have the following proposition.

\begin{proposition} \label{EM Formula}
	For each $N \geq 1$, as $z \to 0$ in the region $D_\theta$, we have
	\begin{align*}
		\sum_{m \geq 0} f\lp (m+1)z \rp = \dfrac{1}{z} \int_0^\infty f(x) dx &- \sum_{k \geq 0} \dfrac{f^{(k)}(0) z^k}{(k+1)!} - \sum_{n=0}^{N-1} \dfrac{B_{n+1}(0) f^{(n)}(z)}{(n+1)!} z^n \\ &\quad- \dfrac{(-1)^N z^{N-1}}{N!} \int_z^{z\infty} f^{(N)}(w) \widetilde{B}_N\lp \frac{w}{z} - 1 \rp dw,
	\end{align*}
	when $f$ and all its derivatives have sufficient decay at infinity, where the last integral is taken along a path of fixed argument.
\end{proposition}

\subsection{Wright's circle method}

We now recall Wright's circle method, which we will use in Section \ref{Asymptotics} to obtain asymptotics for the sequences $a_h(n)$ and $b_h(n)$ as $n \to \infty$. We will use the following special case of the formulation given by Ngo and Rhoades \cite[Proposition 1.8]{NR}.

\begin{proposition} \label{WrightCircleMethod}
    Suppose that $c(n)$ are integers defined by
    \begin{align*}
        \sum_{n \geq 0} c(n) q^n = L(q) \xi(q)
    \end{align*}
    for analytic functions $L, \xi$ satisfying the following hypotheses:
    \begin{itemize}
        \item[(H1)] As $|z| \to 0$ in the cone $D_\theta$, we have $$L\lp e^{-z} \rp \sim \dfrac{1}{z} \sum_{k \geq 0} a_k z^k$$ for $a_k \in \C$,
        \item[(H2)] As $|z| \to 0$ in the cone $D_\theta$, we have $$\xi\lp e^{-z} \rp = K e^{\frac{A}{z}} \lp 1 + O_\theta\lp e^{- \frac{B}{z}} \rp \rp$$ for $K, A\geq0$ and $B>0$\footnote{Ngo and Rhoades require that $B>A$, which arises naturally in the case that $\xi$ is modular but is not strictly necessary when proving Proposition \ref{WrightCircleMethod}.},
        \item[(H3)] As $|z| \to 0$ outside $D_\theta$ and satisfying $\mathrm{Re}(z)>0$, we have $$\left| L\lp e^{-z} \rp \right| \ll_\theta |z|^{-C}$$ for some $C>0$,
        \item[(H4)] As $|z| \to 0$ outside $D_\theta$ and satisfying $\mathrm{Re}(z)>0$, we have $$\left| \xi\lp e^{-z} \rp \right| \ll_\theta \xi\lp \left| e^{-z} \right| \rp e^{- \frac{\delta'}{\mathrm{Re}(z)}}$$ for some $\delta' > 0$.
    \end{itemize}
     Then as $n \to \infty$, we have for any $N \in \Z^+$ that
	\begin{align*}
		c(n) = Ke^{2 \sqrt{A n}} n^{-\frac{1}{4}} \lp \sum\limits_{r=0}^{N-1} p_r n^{-\frac{r}{2}} + O\left(n^{-\frac N2}\right) \rp,
	\end{align*}
	where $p_r := \sum\limits_{j=0}^r a_j c_{j,r-j}$ with $c_{j,r} := \dfrac{(-\frac{1}{4\sqrt{A}})^r \sqrt{A}^{j - \frac 12}}{2\sqrt{\pi}} \dfrac{\Gamma(j + \frac 12 + r)}{r! \Gamma(j + \frac 12 - r)}$. 
\end{proposition}

\begin{remark}
    We note that because our generating functions are given as the generating function for partitions into odd parts multiplied by a rational or nearly rational function, there are methods which in principle can give better error terms. See for example \cite{GKW} for the case of unrestricted partitions.
\end{remark}

Note that hypotheses (H1) and (H2) require asymptotics for $L$ and $\xi$ on the major arc, for $q$ near 1, while hypotheses (H3) and (H4) require that $L$ and $\xi$ are small on the minor arc, for $q$ away from 1.

\section{Asymptotics for $a_h(n)$ and $b_h(n)$} \label{Asymptotics}

In this section, we use Wright's circle method to compute the first terms of the asymptotic expansions of the $q$-series $a_h(n)$ and $b_h(n)$ for integers $h \geq 1$, and we use this to prove Theorems \ref{Alpha Asymptotic Theorem} and \ref{Beta Asymptotic Theorem}. In order to accomplish this, we use Euler--Maclaurin summation to compute asymptotics for $\tilde{a}_h(q)$ and $\tilde{b}_h(q)$ with $q = e^{-z}$, as $z \to 0$ in any conical region (i.e. if $z = x + iy$, we may restrict $z$ to the region $0 \leq x < \delta y$ for any $\delta > 0$).

Although we showed in Section \ref{Generating Functions} that the generating functions of $\tilde{a}_h(q)$ and $\tilde{b}_h(q)$ are rational (or nearly rational) functions multiplied by the modular form $\lp -q;q \rp_\infty$, we use different forms of the generating functions here that are more convenient for Euler--Maclaurin summation. For $\tilde{a}_h(q)$, we shall use the representation from \eqref{ah_for_asymptotic}:
\begin{align*}
	\tilde{a}_h(q) = \lp -q;q \rp_\infty \left[ \sum_{j=0}^{\lceil h/2\rceil-1} \binom{h-j-1}{j}_{q^2} A_1\lp h,j; q \rp + \sum_{j=0}^{\lfloor h/2\rfloor-1} \binom{h-j-2}{j}_{q^2} A_2\lp h,j; q \rp \right]
\end{align*}
 where $A_1(h,j;q)$ and $A_2(h,j;q)$ are defined in \eqref{A_1} and \eqref{A_2}, respectively.
 In both cases, the asymptotic analysis will be carried out using the more general functions
\begin{align*}
	F_{j,k,l}(q) := \sum_{m \geq 0} \lp q^{2m+k}; q^2 \rp_j q^{lm}.
\end{align*}
Observe that $A_1\lp h,j;q \rp = q^h F_{j,1,2h-4j}(q)$ and $A_2\lp h,j;q \rp = q^{3h-4j-2} F_{j,3,2h-4j-2}(q)$, and therefore if $q = e^{-z}$, then as $z\to0$ we have $A_1\lp h,j;q \rp \sim F_{j,1,2h-4j}(e^{-z})$ and $A_2\lp h,j;q \rp \sim F_{j,3,2h-4j-2}(e^{-z})$. Thus if $q=e^{-z}$, then as $z \to 0$ we have the asymptotic formula
\begin{align} \label{a_h first asymptotic}
	\tilde{a}_h(q)&\sim \lp -q;q \rp_\infty \Bigg[ \sum_{j=0}^{\lceil h/2\rceil-1} \binom{h-j-1}{j} F_{j,1,2h-4j}(e^{-z})\nonumber\\
 &\qquad\qquad\qquad+ \sum_{j=0}^{\lfloor h/2\rfloor-1} \binom{h-j-2}{j} F_{j,3,2h-4j-2}(e^{-z}) \Bigg].
\end{align}

For the asymptotic analysis of $\tilde{b}_h(q)$, we will use the representation from \eqref{bh_for_asymptotic}:
\begin{align*}
	\tilde{b}_h(q) = \lp -q;q \rp_\infty B_h(q), \ \ \ B_h(q) := \sum_{j=0}^{\lceil h/2\rceil-1} q^{h + \frac{j(j-1)}{2}} \binom{h-j-1}{j}_q \sum_{m \geq 0} \dfrac{q^{(j+1)m}}{\lp -q^{m+1}; q \rp_{h-j}}.
\end{align*}
In order to understand the asymptotic behavior of $B_h(e^{-z})$ as $z \to 0$, we study the more general functions
\begin{align*}
	G_{j,k}(q) := \sum_{m \geq 0} \dfrac{q^{jm}}{\lp -q^{m+1}; q \rp_k}.
\end{align*}
Observe that if $q = e^{-z}$, then as $z \to 0$ we have the asymptotic formula
\begin{align} \label{b_h first asymptotic}
	\tilde{b}_h(q) \sim \lp -q;q \rp_\infty \sum_{j=0}^{\lceil h/2\rceil-1} \binom{h-j-1}{j} G_{j+1,h-j}(e^{-z}).
\end{align}

We now estimate the functions $F_{j,k,l}(q)$ and $G_{j,k}(q)$, and we apply these estimates to the asymptotic formulas in \eqref{a_h first asymptotic} and \eqref{b_h first asymptotic}. We apply Proposition \ref{EM Formula} to compute asymptotic expansions of the functions $F_{j,k,l}(q)$ and $G_{j,k}(q)$. 

\begin{proposition}\label{F_abc Asymptotics}
As $z\to0$, we have
\begin{align*}
	F_{j,k,l}\lp e^{-z} \rp \sim \dfrac{1}{z}\int_0^\infty \lp 1 - e^{-2x} \rp^j e^{-lx} dx.
\end{align*}
\end{proposition}

\begin{proof}
We use Proposition \ref{EM Formula} to prove Proposition \ref{F_abc Asymptotics}.  It is not quite obvious how to put $F_{j,k,l}(q)$ into the appropriate form, but this can be fixed by introducing a second auxiliary variable as in \cite{BCM}.

Let $t$ be an arbitrary complex number in some region $D_\theta$ as defined in Proposition \ref{EM Formula}. Define the function
\begin{align*}
	F_{j,k,l}\lp t;z \rp := \sum_{m \geq 0} \lp e^{-2mz -kt}; e^{-2t} \rp_j e^{-lmz}.
\end{align*}
Observe that $F_{j,k,l}\lp z;z \rp = F_{j,k,l}(e^{-z})$. Furthermore, if we define
\begin{align*}
	f_{j,k,l}(t;z) := \lp e^{-2z-kt}; e^{-2t} \rp_j e^{-lz},
\end{align*}
we can see that
\begin{align*}
	F_{j,k,l}(t;z) = \sum_{m \geq 0} f_{j,k,l}(t;mz) = f_{j,k,l}(t;0) + \sum_{m \geq 0} f_{j,k,l}\lp t; (m+1)z \rp.
\end{align*}
Observe now that for any fixed $t$ in $D_\theta$, $f_{j,k,l}(t;z)$ is a holomorphic function of $z$ in $D_\theta$ and can be written as a linear combination of exponential decay functions. Therefore, it is clear that $f_{j,k,l}(t;z)$ satisfies the criterion of Proposition \ref{EM Formula} for all fixed $t \in D_\theta$, and therefore
\begin{align*}
	\sum_{m \geq 0} f_{j,k,l}\lp t;(m+1)z \rp &= \dfrac{1}{z} \int_0^\infty f_{j,k,l}(t;x) dx\\
 &\qquad- \sum_{m \geq 0} \dfrac{f_{j,k,l}^{(m)}(t;0) z^m}{(m+1)!} - \sum_{n=0}^{N-1} \dfrac{B_{n+1}(0) f_{j,k,l}^{(n)}(t;z)}{(n+1)!} z^n\\
 &\qquad- \dfrac{(-1)^N z^{N-1}}{N!} \int_z^{z\infty} f_{j,k,l}^{(N)}(t;w) \widetilde{B}_N\lp \frac{w}{z} - 1 \rp dw.
\end{align*}
Since we restrict to $z \to 0$ in the region $D_\theta$, we obtain
\begin{align*}
	F_{j,k,l}\lp e^{-z} \rp = F_{j,k,l}\lp z;z \rp &= f_{j,k,l}(z;0) + \dfrac{1}{z} \int_0^\infty f_{j,k,l}(z;x) dx\\
 &\qquad- \sum_{m \geq 0} \dfrac{f_{j,k,l}^{(m)}(z;0) z^m}{(m+1)!} - \sum_{n=0}^{N-1} \dfrac{B_{n+1}(0) f_{j,k,l}^{(n)}(z;z)}{(n+1)!} z^n \\ &\qquad- \dfrac{(-1)^N z^{N-1}}{N!} \int_z^{z\infty} f_{j,k,l}^{(N)}(z;w) \widetilde{B}_N\lp \frac{w}{z} - 1 \rp dw.
\end{align*}
Now, observe that $f_{j,k,l}(t;z)$ is holomorphic at $t=0$ and therefore the identification $t = z$ does not introduce any additional singularities at $z=0$. Furthermore, because $f_{j,k,l}(t;z)$ is holomorphic at zero for both the $t$ and $z$ variables, it is easy to see that the only term in this expression which contributes to the principal part as $z \to 0$ is $\frac{1}{z} \int_0^\infty f_{j,k,l}(z;x)dx$, and therefore as $z \to 0$ in $D_\theta$ we obtain
\begin{align*}
	F_{j,k,l}\lp e^{-z} \rp \sim \dfrac{1}{z} \int_0^\infty f_{j,k,l}(z;x) dx.
\end{align*}
It is straightforward to see analytically that $\lim_{z \to 0} \int_0^\infty f_{j,k,l}(z;x) dx = \int_0^\infty f_{j,k,l}(0;x) dx$, and therefore as $z \to 0$ in $D_\theta$ we obtain
\begin{align*}
	F_{j,k,l}\lp e^{-z} \rp \sim \dfrac{1}{z} \int_0^\infty f_{j,k,l}(0;x) dx.
\end{align*}
The result follows.
\end{proof}

We define for convenience the integrals
\begin{align}\label{I_A Definition}
	I_A(j,l) := \int_0^\infty f_{j,k,l}(0;x) dx = \int_0^\infty \lp 1 - e^{-2x} \rp^j e^{-lx} dx.
\end{align}
Observe that the notation $I_A(j,l)$ is sufficient because the resulting integral does not depend on $k$. Combining these calculations with \eqref{a_h first asymptotic}, it follows that for $q = e^{-z}$, as $z \to 0$ in $D_\theta$, we have
\begin{align} \label{a_h second asymptotic}
	\tilde{a}_h\lp q\rp &\sim \dfrac{\lp -q;q \rp_\infty}{z} \Bigg[ \sum_{j=0}^{\lceil h/2\rceil-1} \binom{h-j-1}{j} I_A(j,2h-4j)\nonumber\\
 &\qquad\qquad\qquad+ \sum_{j=0}^{\lfloor h/2\rfloor-1} \binom{h-j-2}{j} I_A(j,2h-4j-2) \Bigg].
\end{align}

We follow a very similar process to estimate $G_{j,k}(q)$, and we obtain the following asymptotic formula.

\begin{proposition}\label{G_ab Asymptotics}
As $z\to0$, we have
\begin{align*}
	G_{j,k}\lp e^{-z} \rp \sim \dfrac{1}{z}\int_0^\infty \dfrac{e^{-jx}}{\lp 1 + e^{-x} \rp^k} dx.
\end{align*}
\end{proposition}

\begin{proof}
As before, we restrict $t,z$ to a region $D_\theta$, $q = e^{-z}$. Consider
\begin{align*}
	g_{j,k}(t;z) := \dfrac{e^{-jz}}{\lp e^{-z-t}; e^{-t} \rp_k}.
\end{align*}
Then we have
\begin{align*}
	G_{j,k}\lp e^{-z} \rp = \sum_{m \geq 0} g_{j,k}(z;mz) = g_{j,k}(z;0) + \sum_{m \geq 0} g_{j,k}\lp z; (m+1)z \rp.
\end{align*}
Using Proposition \ref{EM Formula} with $t$ fixed as in the previous case, we obtain for
\begin{align*}
	G_{j,k}(t;z) := \sum_{m \geq 0} g_{j,k}(t;mz)
\end{align*}
the identity
\begin{align*}
	G_{j,k}(t;z) &= g_{j,k}(t;0) + \dfrac{1}{z} \int_0^\infty g_{j,k}(t;x) dx - \sum_{m \geq 0} \dfrac{g_{j,k}^{(m)}(t;0) z^m}{(m+1)!}\\
 &\quad- \sum_{n=0}^{N-1} \dfrac{B_{n+1}(0) g_{j,k}^{(n)}(t;z)}{(n+1)!} z^n - \dfrac{(-1)^N z^{N-1}}{N!} \int_z^{z\infty} g_{j,j}^{(N)}(t;w) \widetilde{B}_N\lp \frac{w}{z} - 1 \rp dw,
\end{align*}
valid for any $N \geq 1$. As before, the holomorphicity properties of $g_{j,k}(t;z)$ in the $z$ and $t$ variables imply that as $z \to 0$ in $D_\theta$, we have
\begin{align*}
	G_{j,k}\lp e^{-z} \rp \sim G_{j,k}(z;z) \sim \dfrac{1}{z} \int g_{j,k}(z;x) dx \sim \dfrac{1}{z} \int_0^\infty g_{j,k}(0;x) dx.
\end{align*}
The result follows.
\end{proof}

If we define the integrals
\begin{align}\label{I_B Definition}
	I_B(j,k) := \int_0^\infty g_{j,k}(0;x) dx = \int_0^\infty \dfrac{e^{-jx}}{\lp 1 + e^{-x} \rp^k} dx,
\end{align}
then for $q=e^{-z}$, as $z\to0$ in $D_\theta$, we obtain
\begin{align} \label{b_h second asymptotic}
	\tilde{b}_h(q) \sim \dfrac{\lp -q;q \rp_\infty}{z} \sum_{j=0}^{\lceil h/2\rceil-1} \binom{h-j-1}{j} I_B(j+1,h-j).
\end{align}

To summarize our asymptotic analysis so far, we have the following asymptotic formulas for $\tilde{a}_h(q)$ and $\tilde{b}_h(q)$ from Propositions \ref{F_abc Asymptotics} and \ref{G_ab Asymptotics}.

\begin{theorem}\label{a_h and b_h Asymptotic Theorem}
For $q=e^{-z}$, as $z\to0$ we have
\begin{align*}
\tilde{a}_h(q)\sim\frac{\alpha_h}{z}\lp-q;q\rp_\infty\quad\text{and}\quad \tilde{b}_h(q)\sim\frac{\beta_h}{z}\lp-q;q\rp_\infty,
\end{align*}
where $\alpha_h$ and $\beta_h$ are constants defined by
\begin{align}\label{alpha_h_constant}
	\alpha_h := \sum_{j=0}^{\lceil h/2\rceil-1} \binom{h-j-1}{j} I_A(j,2h-4j) + \sum_{j=0}^{\lfloor h/2\rfloor-1} \binom{h-j-2}{j} I_A(j,2h-4j-2)
\end{align}
and
\begin{align}\label{beta_h_constant}
	\beta_h := \sum_{j=0}^{\lceil h/2\rceil-1} \binom{h-j-1}{j} I_B(j+1,h-j),
\end{align}
and $I_A(j,k)$ and $I_B(j,k)$ are defined in \eqref{I_A Definition} and \eqref{I_B Definition}, respectively.
\end{theorem}

Since we have shown in Section \ref{Generating Functions} that $\tilde{a}_h(q)/\lp -q;q \rp_\infty$ and $\tilde{b}_h(q)/\lp -q;q \rp_\infty$ are essentially rational functions (in particular, when $q = e^{-z}$, they only have polar singularities as $z \to 0$), we can now apply Wright's circle method, which we recalled in Subsection \ref{WrightCircleMethod}, 
to obtain an asymptotic expansion for the coefficients $a_h(n)$ and $b_h(n)$. Thus, we are now in a position to prove Theorems \ref{Alpha Asymptotic Theorem} and \ref{Beta Asymptotic Theorem}.

\begin{proof}[Proofs of Theorems \ref{Alpha Asymptotic Theorem} and \ref{Beta Asymptotic Theorem}]
From the modularity of the Dedekind eta function $\eta(z):=q^{1/24}\prod_{n\geq1}\left(1-q^n\right)$, we obtain the transformation law $\eta(-1/z)=\sqrt{-iz}\eta(z)$, from which it is easy to show using Proposition \ref{WrightCircleMethod} (see \cite{BBCFW} for more details) that for $q = e^{-z}$ we have
\begin{align} \label{Odd GF Asymptotic}
    \lp -q;q \rp_\infty = \dfrac{1}{\sqrt{2}} \exp\lp \dfrac{\pi^2}{12z} \rp \lp 1 + O_\delta\lp e^{-\frac{B}{z}} \rp \rp
\end{align}
as $z \to 0$ inside $D_\theta$. If $z=x+iy$, then outside of $D_\theta$ we have that
\begin{align*}
|\xi(q)|\ll_\theta\xi\left(e^{-x}\right)\cdot\exp\left(-\frac{\delta'}{x}\right)
\end{align*}
for some $\delta'>0$.  More specifically, \cite[Lemmas 5.9 and 5.10]{BBCFW} show that $\lp -q;q \rp_\infty$ satisfies hypotheses (H2) and (H4) with $K = \frac{1}{\sqrt{2}}$, $A = \frac{\pi^2}{12}$ for both $\tilde{a}_h(q)$ and $\tilde{b}_h(q)$.  We calculate that
\begin{equation*}
p_0 = \alpha_hc_{0,0} = \alpha_h\frac{3^{1/4}}{\pi \sqrt{2}} \quad\text{and}\quad p_0 = \beta_hc_{0,0} = \beta_h\frac{3^{1/4}}{\pi \sqrt{2}}
\end{equation*} for $\tilde{a}_h(q)$ and $\tilde{b}_h(q)$, respectively. Therefore, we have
\begin{align*}
    a_h(n) \sim \alpha_h \dfrac{3^{1/4}}{2\pi n^{1/4}} e^{\pi \sqrt{\frac{n}{3}}}\quad\text{and}\quad b_h(n) \sim \beta_h \dfrac{3^{1/4}}{2\pi n^{1/4}} e^{\pi \sqrt{\frac{n}{3}}},
\end{align*}
which completes the proofs of the theorems.
\end{proof}

\begin{remark}
Note that in order to compute the constants $N_h$ in Theorem \ref{Main Theorem}, one could use the explicit bounds in \cite{JO}.  To do this, one must first make the asymptotics for the rational functions $\tilde{A}_h(q)$ and $\tilde{B}_h(q)$ from Theorem \ref{th:rat:ab} effective, which could potentially be done by either an effective Taylor theorem or effective Euler--Maclaurin summation as used in \cite{C,JO}.
\end{remark}

We now see that the inequality $a_h(n) > b_h(n)$ will follow for $n \gg 0$ if $\alpha_h > \beta_h$.

\section{Evaluating $\alpha_h$ and $\beta_h$} \label{Constants}

In this section, we complete the proof of Theorem \ref{Main Theorem} by showing that $\alpha_h > \beta_h$ for all $h \geq 2$. Recall that $\alpha_h$ and $\beta_h$ are defined in \eqref{alpha_h_constant} and \eqref{beta_h_constant} respectively, and the integrals $I_A(j,k)$ and $I_B(j,k)$ are defined in \eqref{I_A Definition} and \eqref{I_B Definition}. In order to prove that $\alpha_h > \beta_h$ for all $h \geq 2$, we proceed in stages. First, we produce simpler formulas for $\alpha_h$ and $\beta_h$ which involve harmonic numbers. We then leverage these simpler formulas to prove that $\alpha_h > \beta_h$ for $h \geq 2$. In fact, we will prove something much stronger:

\begin{theorem} \label{Limit Values and Inequality}
    The following are true:
    \begin{itemize}
        \item[(1)] We have $\alpha_h \to \log(2)$ and $\beta_h \to \frac{\log(3)}{2}$ as $h \to \infty$.
        \item[(2)] We have $\alpha_h > \beta_h$ for all $h \geq 2$.
    \end{itemize}
\end{theorem}

Observe that Theorems \ref{Alpha Asymptotic Theorem}, \ref{Beta Asymptotic Theorem}, and \ref{Limit Values and Inequality} together imply Theorem \ref{Inequality Theorem}, which then completes the proof of Theorem \ref{Main Theorem}. Therefore, we focus the remainder of this section on the proof of Theorem \ref{Limit Values and Inequality}.

To prove this theorem, we prove a sequence of lemmas that give successively simpler values for $\alpha_h$ and $\beta_h$. To help with simplifying $\alpha_h$, we define
\begin{align} \label{Alpha-tilde definition}
    \widetilde \alpha_h := \sum_{j=0}^{\lfloor \frac{h-1}{2} \rfloor}\sum_{k=0}^j {h-j-1\choose j}{j\choose k}\frac{(-1)^k}{h-2j+k}.
\end{align}

\begin{lemma} \label{Alpha Integrals Lemma}
    For any integer $h \geq 1$, we have
    \begin{align*}
        \alpha_h = \dfrac{\widetilde \alpha_h + \widetilde \alpha_{h-1}}{2}.
    \end{align*}
\end{lemma}

\begin{proof}
    By checking the changes induced by taking $h \mapsto h+1$, the result will follow from the definition of $\alpha_h$ if we can show that
    \begin{align*}
        \dfrac{\widetilde \alpha_h}{2} = \sum_{j=0}^{\lfloor \frac{h-1}{2} \rfloor} \binom{h-j-1}{j} I_A(j,2h-4j),
    \end{align*}
    which in turn follows if we prove that
    \begin{align*}
        I_A\lp j,2h-4j \rp = \dfrac{1}{2} \sum_{k=0}^j \binom{j}{k} \dfrac{(-1)^k}{h-2j+k}.
    \end{align*}
    Now, from the definition of $I_A(j,l)$ and by the substitution $u = e^{-2x}$, we have
    \begin{align*}
        I_A\lp j, 2h-4j \rp = \int_0^\infty \lp 1 - e^{-2x} \rp^a e^{-2\lp h-2j \rp x} dx = \dfrac{(-1)^j}{2} \int_0^1 \lp u-1 \rp^j u^{h-2j-1} du.
    \end{align*}
    The result then follows by expanding $\lp u-1 \rp^j$ with the binomial theorem and integrating.
\end{proof}

\begin{lemma} \label{Beta Integrals Lemma}
    Let $h \geq 1$ be a positive integer. Then we have
    \begin{align*}
        \beta_h = \sum_{j=0}^{\frac{h-2}{2}} \binom{h-j-1}{j} \sum_{k=0}^j \binom{j}{k} \dfrac{(-1)^k}{h-2j+k-1}\lp 1 - \dfrac{1}{2^{h-2j+k-1}} \rp,
    \end{align*}
    whenever $h$ is even, and
    \begin{align*}
        \beta_h = \log(2) &+ \sum_{j=0}^{\frac{h-1}{2}} \binom{h-j-1}{j} \sum_{k=1}^j \binom{j}{k} \dfrac{(-1)^k}{h-2j+k-1}\lp 1 - \dfrac{1}{2^{h-2j+k-1}} \rp \\ &\qquad+ \sum_{j=0}^{\frac{h-3}{2}} \binom{h-j-1}{j} \dfrac{1}{h-2j-1}\lp 1 - \dfrac{1}{2^{h-2j-1}} \rp,
    \end{align*}
    whenever $h$ is odd.
\end{lemma}

\begin{proof}
    From the definition of $I_B(j,k)$, we have
    \begin{align*}
        \beta_h = \sum_{j=0}^{\lfloor \frac{h-1}{2} \rfloor} \binom{h-j-1}{j} \int_0^\infty \dfrac{e^{-(j+1)x}}{\lp 1 + e^{-x} \rp^{h-j}} dx.
    \end{align*}
    Using the substitution $u = 1 + e^{-x}$ and expanding the resulting power of $(u-1)$ with the binomial theorem, we have
    \begin{align*}
        \int_0^\infty \dfrac{e^{-(j+1)x}}{\lp 1 + e^{-x} \rp^{h-j}} dx = \int_1^2 \dfrac{\lp u - 1 \rp^j}{u^{h-j}} du = \sum_{k=0}^j (-1)^k \binom{j}{k} \int_1^2 u^{2j-h-k} du.
    \end{align*}
    Now, considering the limitations $0 \leq k \leq j$ and $0 \leq j \leq \lfloor \frac{h-1}{2} \rfloor$, we have $-h \leq 2j-h-k \leq -1$, with equality to $-1$ if and only if $j = \frac{h-1}{2}$ and $k = 0$. Note that this scenario is only possible if $h$ is odd. Thus, we have
    \begin{align*}
        \int_1^2 u^{2j-h-k}du = \begin{cases}
            \dfrac{1}{h-2j+k-1}\lp 1 - \dfrac{1}{2^{h-2j+k-1}} \rp & \text{if } \lp j, k \rp \not = \lp \frac{h-1}{2}, 0 \rp, \\ \log(2) & \text{if } \lp j, k \rp = \lp \frac{h-1}{2}, 0 \rp.
        \end{cases}
    \end{align*}
    It is therefore convenient to split into cases. Firstly, if $h$ is even, then the $\log(2)$ term does not emerge and we obtain
    \begin{align*}
        \beta_h = \sum_{j=0}^{\frac{h-2}{2}} \binom{h-j-1}{j} \sum_{k=0}^j \binom{j}{k} \dfrac{(-1)^k}{h-2j+k-1}\lp 1 - \dfrac{1}{2^{h-2j+k-1}} \rp.
    \end{align*}
    If $h$ is odd on the other hand, we must isolate the term $\lp j,k \rp = \lp \frac{h-1}{2}, 0 \rp$, and so we obtain
    \begin{align*}
        \beta_h = \log(2) &+ \sum_{j=0}^{\frac{h-1}{2}} \binom{h-j-1}{j} \sum_{k=1}^j \binom{j}{k} \dfrac{(-1)^k}{h-2j+k-1}\lp 1 - \dfrac{1}{2^{h-2j+k-1}} \rp \\ &\qquad+ \sum_{j=0}^{\frac{h-3}{2}} \binom{h-j-1}{j} \dfrac{1}{h-2j-1}\lp 1 - \dfrac{1}{2^{h-2j-1}} \rp.
    \end{align*}
    This completes the proof.
\end{proof}

Lemmas \ref{Alpha Integrals Lemma} and \ref{Beta Integrals Lemma} simplify our considerations to the evaluation of linear combinations of merely rational numbers. We will analyze these finite sums using the following lemmas studying related polynomials.

\begin{lemma} \label{Polynomial Lemma}
    Define the polynomials
    \begin{align*}
        F_{n,m}(x) := \sum_{k=0}^n\binom{n}{k}\frac{(-1)^k x^{m+k}}{m+k}, \quad R_n(x) := \sum_{k=1}^{n}{n\choose k}\frac{(-1)^k x^k}{k}.
    \end{align*}
    Then the following are true:
    \begin{itemize}
        \item[(1)] We have $$F_{n,m}(x) = \frac{x^m \sum_{j=0}^{n+m-1} (1-x)^{n-j}  \binom{n+m-1-j}{m-1}}{m\binom{n+m}{m}}.$$
        \item[(2)] We have $$R_n(x) = \int_0^x \dfrac{(1-t)^n-1}{t} dt.$$
    \end{itemize}
\end{lemma}

\begin{proof}
    We first observe that
    \begin{align}\label{FnmIntermediate}
        F_{n,m}^\prime(x) = \sum_{k=0}^n \binom{n}{k} (-1)^k x^{m+k-1} = x^{m - 1} \sum_{k=0}^n \binom{n}{k} (-1)^k x^k = x^{m-1}(1-x)^n.
    \end{align}
    To prove (1), we compute the generating function for $F_{n,m}(x)$. We have by \eqref{FnmIntermediate} that
    \begin{align*}
        \sum_{m\geq 1}\sum_{n\geq 0} F_{n,m}(x) \frac{u^n}{n!}\frac{ v^{m-1}}{(m-1)!}
	  &= \sum_{m\geq 1}\sum_{n\geq 0} \frac{u^n}{n!}\frac{ v^{m-1}}{(m-1)!}
        \int_{0}^x t^{m-1}(1-t)^n dt.
    \end{align*}
    By swapping the order of summation and integration and using the series for $e^t$, we may simplify:
    \begin{align*}
	  \sum_{m\geq 1}\sum_{n\geq 0} F_{n,m}(x) \frac{u^n}{n!}\frac{ v^{m-1}}{(m-1)!} &=       \int_{0}^x e^{vt}e^{(1-t)u} dt\\
	  &= \int_{0}^x e^{u+ t(v-u)} dt\\
	  &= \sum_{k\geq 1} \frac{ (xv-xu+u)^k - u^k}{(v-u) k!}\\
	  &= x\sum_{k\geq 1} \frac{\sum_{j=0}^{k-1} (xv-xu+u)^{k-1-j} u^j }{k!}.
    \end{align*}
    Therefore, using the notation $[x^n] F(x)$ to denote the coefficient of $x^n$ in the expression $F(x)$, we obtain
    \begin{align*}
        F_{n,m}(x) 
	  &= [u^n v^{m-1}]\ n!(m-1)! x\sum_{k\geq 1} \frac{\sum_{j=0}^{k-1} (xv-xu+u)^{k-1-j}     u^j }{k!}\\
	  &= [u^n v^{m-1}]\ n!(m-1)! x\sum_{k\geq 1} \frac{\sum_{j=0}^{k-1} (xv)^{m-1}((1-        x)u)^{k-j-m} \binom{k-1-j}{m-1}  u^j }{k!}\\
	  &= [u^n ]\ n!(m-1)!x^m\sum_{k\geq 1} \frac{\sum_{j=0}^{k-1} (1-x)^{k-j-m} u^{k-m} 
        \binom{k-1-j}{m-1}   }{k!}\\
	  &= n!(m-1)! x^m \frac{\sum_{j=0}^{n+m-1} (1-x)^{n-j}  \binom{n+m-1-j}{m-1}}{(n+m)!}\\
	  &= \frac{x^m \sum_{j=0}^{n+m-1} (1-x)^{n-j}  \binom{n+m-1-j}{m-1}   }{m\binom{n+m}{m}}
    \end{align*}
    This completes the proof of (1). To prove (2), it is enough to observe that
    \begin{equation*}
        R_n^\prime(x) = \sum_{k=1}^{n}{n\choose k}{(-1)^k}x^{k-1} = \dfrac{(1-x)^n-1}{x}.\qedhere
    \end{equation*}
\end{proof}

We also require the following lemma for evaluating a different type of summation.

\begin{lemma} \label{Summation identity}
    For $n=2m$ even, we have
    \begin{align*}
        \sum_{k=0}^m \dfrac{1}{n-2k+1} \binom{n-k}{k}{2^k} = \dfrac{1}{n+1} \lp {2^{n+1} -1} \rp.
    \end{align*}
    For $n=2m+1$ odd, we have
    \begin{align*}
        \sum_{k=0}^m \dfrac{1}{n-2k+1} \binom{n-k}{k} {2^k} = \dfrac{1}{n+1} \lp {2^{m+1} -1} \rp^2.
    \end{align*}
\end{lemma}

\begin{proof}
    Firstly, we observe the fact that
    \begin{align*}
        \sum_{n\geq 0} \sum_{k=0}^n \binom{n-k}{k} {2^k} x^n = \sum_{k\geq 0} {(2x^2)^k}\sum_{n\geq 2k} \binom{n-k}{k}  x^{n -2k} &= \sum_{k\geq 0} {(2x^2)^k} \frac{1}{(1-x)^{k+1}} \\ &= \dfrac{1}{3}\lp \frac{2}{1-2x} +  \frac{1}{1+x} \rp.
    \end{align*}
    This generating function identity implies that
    \begin{align*}
        \sum_{n\geq 0} \sum_{k=0}^n \binom{n-k}{k} {2^k} = \dfrac{2^{n+1} - (-1)^{n+1}}{3}.
    \end{align*}
    To continue the proof, consider
    \begin{align*}
        S(m) := \sum_{k=0}^m \frac{1}{n-2k+1} \binom{n-k}{k} {2^k}.
    \end{align*}
    Using the previous identity, we have
    \begin{align*}
        S(n) &= \frac{1}{n+1} \sum_{k=0}^m  \binom{n-k}{k} {2^k} +\frac{1}{n+1} \sum_{k=0}^m \frac{ 2k}{n-2k+1} \binom{n-k}{k} {2^k}\\
        &=\frac{2^{n+1} -(-1)^{n+1}}{3(n+1)} + \frac{2 T(n)}{n+1},
    \end{align*}
    where (for $n = 2m$ or $2m+1$)
    \begin{align*}
        T(n) &= \sum_{k=0}^m \frac{ k}{n-2k+1} \binom{n-k}{k} {2^k} = \sum_{k=1}^m \binom{n-k}{k-1} {2^k} \\ &= 2 \sum_{k=1}^m \binom{n-1-(k-1)}{k-1} {2^{k-1}} = 2 \sum_{k=0}^{m-1} \binom{n-1-k}{k} {2^{k}}.
    \end{align*}
    If $n = 2m$ is even, we have
    \begin{align*}
        T(n) = 2\lp \frac{2^{n} -(-1)^{n}}{3} \rp,
    \end{align*}
    and if $n = 2m+1$ is odd, we have
    \begin{align*}
        T(n) = 2 \lp \sum_{k=0}^{m} \binom{n-1-k}{k} 2^k - 2^m \rp = 2\lp \frac{2^{n} -(-1)^{n}}{3} \rp - {2^{m+1}}.
    \end{align*}
    Thus, if $n = 2m$ is even, we have
    \begin{align*}
        S(n)
        &=\frac{2^{n+1} -(-1)^{n+1}}{3(n+1)}  + \frac{2T(n)}{n+1} \\
        &= \frac{2^{n+1} -(-1)^{n+1}}{3(n+1)}  + \frac{4}{n+1} \frac{2^{n} -(-1)^{n}}{3}
        =\frac{1}{n+1}\left( {2^{n+1} -1} \right),
    \end{align*}
    and if $n = 2m+1$ is odd, we have
    \begin{align*}
        S(n)
        &=\frac{2^{n+1} -(-1)^{n+1}}{3(n+1)}  + \frac{2T(n)}{n+1} = S(n-1) - \frac{2\cdot 2^{m+1}}{n+1}\\
        &=\frac{1}{2m+2}\left( {2^{2m+2} +1} - {2^{m+2}} \right) = \frac{1}{2m+2}\left( {2^{m+1} -1}  \right)^2.
\end{align*}
This completes the proof.
\end{proof} 

With the aid of Lemma \ref{Polynomial Lemma}, we now further simplify the formulas for $\alpha_h$ and $\beta_h$ given in Lemmas \ref{Alpha Integrals Lemma} and \ref{Beta Integrals Lemma}. For the following result, we need the well-known {\it harmonic numbers}, defined by $H_0 = 0$ and
\begin{align*}
    H_n := \sum_{k=1}^n \frac{1}{k}
\end{align*}
for $n \geq 1$.

\begin{theorem} \label{Alpha Sum Formula}
    We have for all $h \geq 1$ that
    \begin{align*}
        \alpha_h = \dfrac{H_h - H_{\lceil \frac{h-1}{2} \rceil} + H_{h-1} - H_{\lceil \frac{h-2}{2} \rceil}}{2}.
    \end{align*}
\end{theorem}

\begin{proof}
    By Lemma \ref{Polynomial Lemma} (1), we have the identity
    \begin{align*}
	\sum_{k=0}^n \binom{n}{k} \dfrac{(-1)^k}{m+k} = F_{n,m}(1) = \dfrac{1}{m \binom{m+n}{n}}.
    \end{align*}
    Using this identity with $n = j$ and $m = h - 2j$ we obtain from \eqref{Alpha-tilde definition} that
    \begin{align*}
        \widetilde\alpha_h &= \sum_{j = 0}^{\lfloor \frac{h-1}{2} \rfloor} \binom{h-j-1}{j} \sum_{k=0}^j \binom{j}{k} \dfrac{(-1)^k}{h-2j+k} = \sum_{j=0}^{\lfloor \frac{h-1}{2} \rfloor} \binom{h-j-1}{j} \dfrac{1}{(h-2j) \binom{h-j}{j}} \\ &= \sum_{j=0}^{\lfloor \frac{h-1}{2} \rfloor} \dfrac{1}{h-2j} \left( 1 - \dfrac{j}{h-j} \right) = \sum_{j = 0}^{\lfloor \frac{h-1}{2} \rfloor} \dfrac{1}{h-j}.
    \end{align*}
    Thus, 
    \begin{align}
    \label{alpha_h H formula}
        \widetilde\alpha_h = \begin{cases} H_h - H_{h/2} & \text{ if } 2|h, \\ H_h - H_{\frac{h-1}{2}} & \text{ if } 2 \centernot | h. \end{cases}
    \end{align}
    As it has already been established that $\alpha = \frac{\widetilde\alpha_h + \widetilde\alpha_{h-1}}{2}$, we obtain the desired formula for $\alpha_h$.
\end{proof}

We also note briefly that this theorem implies part of Theorem \ref{Limit Values and Inequality}.

\begin{corollary} \label{Alpha Limit}
    We have $\alpha_h \to \log(2)$ as $h \to \infty$.
\end{corollary}

\begin{proof}
    Recall the well-known fact that
    \begin{align*}
        \lim\limits_{n \to \infty} \lp H_n - \log(n) \rp = \gamma,
    \end{align*}
    where $\gamma$ is the Euler--Mascheroni constant. Thus, $\widetilde\alpha_h \to \log(h) - \log(h/2) = \log(2)$ and $\alpha_h = \frac{\widetilde\alpha_h + \widetilde\alpha_{h-1}}{2} \to \log(2)$.
\end{proof}

\begin{remark}\label{remark5}
We note that a simpler proof of the limiting value $\alpha_h\to\log(2)$ exists.  It is possible to express the integrals $I_a(j,2l)$ as quotients of gamma values as follows:
\begin{align*}
I_A(j,2l)=\frac{\Gamma(j+1)\Gamma(l)}{2\Gamma(j+l+1)}.
\end{align*}
From here, we can obtain the limiting value $\widetilde{\alpha}_h\to\log(2)$, which implies the result.  We provide a longer proof above, because several equations and lemmas proven along the way will be useful in the proof of the limiting value of $\beta_h$ in Theorem \ref{Limit Values and Inequality} (1) and in the proof of Theorem \ref{Limit Values and Inequality} (2).
\end{remark}

We now prove the limiting theorems for $\beta_h$, which is the last major step before we complete the proof of Theorem \ref{Limit Values and Inequality}, and thus also of Theorem \ref{Main Theorem}.

\begin{theorem} \label{Beta Sum Formula}
    We have for all $h \geq 1$ that
    \begin{align*}
        \beta_h = \begin{cases}
            \displaystyle\sum_{c=0}^{\frac{h-2}{2}} \dfrac{1}{(2c+1)2^{2c+1}} & \text{ if } 2|h, \\ \displaystyle\sum_{c=0}^{\frac{h-3}{2}} \dfrac{1}{(c+1) 2^{c+1}} - \displaystyle\sum_{c=0}^{\frac{h-3}{2}} \dfrac{1}{(2c+2) 2^{2c+2}} + \displaystyle\int_{1/2}^1 \dfrac{(1-x)^{\frac{h-1}{2}}}{x} dx & \text{ if } 2 \centernot | h.
        \end{cases}
    \end{align*}
\end{theorem}

\begin{proof}
    We begin first with $h = 2n$ even, in which case we have
    \begin{align*}
        \beta_h = \sum_{j=0}^{n-1} \sum_{k=0}^j \binom{h-j-1}{j} \binom{j}{k} \dfrac{(-1)^k}{h-2j+k-1} \lp 1 - \dfrac{1}{2^{h-2j+k-1}} \rp.
    \end{align*}
    Using Lemma \ref{Polynomial Lemma} (1), we have
    \begin{align*}
        \sum_{j=0}^{n-1} \sum_{k=0}^j \binom{h-j-1}{j} \binom{j}{k} \dfrac{(-1)^k}{h-2j+k-1} = \sum_{j=0}^{n-1} \dfrac{1}{h-2j-1},
    \end{align*}
    and so we may write
    \begin{align} \label{Eq 1}
        \beta_h = \sum_{j=0}^{n-1} \frac{1}{h-2j-1} - U_h,
    \end{align}
    where
    \begin{align*}
        U_h := \sum_{j=0}^{n-1} \frac{1}{h-2j-1} \sum_{k=0}^{h-j-2} \binom{h-j-k-2}{h-2j-2}  \frac{1}{2^{h-j-k-1}}.
    \end{align*}
    We first simplify $U_h$ and swap the order of summation:
    \begin{align*}
        U_h &= \sum_{j=0}^{n-1} \dfrac{1}{h-2j-1} \sum_{k=0}^j \binom{h-2j-2+k}{k} \dfrac{1}{2^{h-2j-1+k}} \\ &= \sum_{k=0}^{n-1} \sum_{j=k}^{n-1} \dfrac{1}{h-2j-1} \binom{h-2j-2+k}{k} \dfrac{1}{2^{h-2j-1+k}}.
    \end{align*}
    Since $h = 2n$, we write $b = n-j$ and obtain
    \begin{align*}
        U_h = \sum_{k=0}^{n-1} \sum_{b=0}^{n-k-1} \dfrac{1}{2b+1} \binom{2b+k}{k} \frac{1}{2^{2b+1+k}} &= \sum_{k=0}^{n-1} \dfrac{1}{2^k} \sum_{b=0}^{n-k-1} \dfrac{1}{2b+1} \binom{2b+k}{k} \dfrac{1}{2^{2b+1}}.
    \end{align*}
    If we further reindex the double sum with $c = b+k$, we have by Lemma \ref{Polynomial Lemma} (3) that
    \begin{align*}
        U_h &= \sum_{c=0}^{n-1} \sum_{k=0}^{c} \dfrac{1}{2c-2k+1} \binom{2c-k}{k}  \dfrac{1}{2^{2c-k+1}} = \sum_{c=0}^{n-1} \dfrac{1}{2c+1} \lp 1 - \dfrac{1}{2^{2c+1}} \rp.
    \end{align*}
    Thus, we may simplify \eqref{Eq 1} for $\beta_h$ to the form
    \begin{align*}
        \beta_h 
        &= \sum_{j=0}^{n-1} \dfrac{1}{h-2j-1} - \sum_{c=0}^{n-1} \dfrac{1}{2c+1} \lp 1 - \dfrac{1}{2^{2c+1}}\rp \\
        &= \sum_{c=0}^{n-1} \dfrac{1}{h-2(n-1-c)-1} 
        - \sum_{c=0}^{n-1} \dfrac{1}{2c+1}  
        + \sum_{c=0}^{n-1} \dfrac{1}{2c+1} \cdot \dfrac{1}{2^{2c+1}} \\        
        &= \sum_{c=0}^{n-1}  \dfrac{1}{(2c+1)2^{2c+1}}.
    \end{align*}
    This proves the theorem in the case where $h$ is even. \\

    Now, assume $h = 2n+1$ is odd. By Lemma \ref{Beta Integrals Lemma}, we have
    \begin{align*}
        \beta_h = \log(2) &+ \sum_{j=0}^{n} \binom{h-j-1}{j} \sum_{k=1}^j \binom{j}{k} \dfrac{(-1)^k}{h-2j+k-1}\lp 1 - \dfrac{1}{2^{h-2j+k-1}} \rp \\ &+ \sum_{j=0}^{n-1} \binom{h-j-1}{j} \dfrac{1}{h-2j-1}\lp 1 - \dfrac{1}{2^{h-2j-1}} \rp.
    \end{align*}
    From Lemma \ref{Polynomial Lemma} (1), we obtain as in the previous case that
    \begin{align*}
        \beta_h = \log(2) + \sum_{j=0}^{n-1} \dfrac{1}{h-2j-1} - U_h + \sum_{k=1}^n \dfrac{(-1)^k}{k}\lp 1 - \dfrac{1}{2^k} \rp.
    \end{align*}
    Now, Lemma \ref{Polynomial Lemma} (2) implies that
    \begin{align*}
        \log(2) + \sum_{k=1}^n \dfrac{(-1)^k}{k} \lp 1 - \dfrac{1}{2^k} \rp = \log(2) + \int_{1/2}^1 \dfrac{\lp 1-x \rp^n - 1}{x} dx = \int_{1/2}^1 \dfrac{(1-x)^n}{x} dx,
    \end{align*}
    and therefore
    \begin{align*}
        \beta_h = \sum_{j=0}^{n-1} \dfrac{1}{h-2j-1} - U_h + \int_{1/2}^1 \dfrac{(1-x)^n}{x} dx.
    \end{align*}
    Now, using Lemma \ref{Summation identity}, we therefore have
    \begin{align*}
        \beta_h &= \sum_{j=0}^{n-1} \dfrac{1}{h-2j-1} - \sum_{c=0}^{n-1} \dfrac{1}{2c+2}\lp 1 - \dfrac{2}{2^{c+1}} + \dfrac{1}{2^{2c+2}} \rp + \int_{1/2}^1 \dfrac{(1-x)^n}{x} dx \\ &= \sum_{c=0}^{n-1} \dfrac{1}{(c+1) 2^{c+1}} - \sum_{c=0}^{n-1} \dfrac{1}{(2c+2) 2^{2c+2}} + \int_{1/2}^1 \dfrac{(1-x)^n}{x} dx.
    \end{align*}
    This completes the proof for $h$ odd.
\end{proof}

We also note briefly that this theorem implies the remaining part of Theorem \ref{Limit Values and Inequality}~(1).

\begin{corollary} \label{Beta Limit}
    We have $\beta_h \to \frac{\log(3)}{2}$ as $h \to \infty$.
\end{corollary}

\begin{proof}
    As $n \to \infty$ it is clear that
    \begin{align*}
        \int_{1/2}^1 \dfrac{(1-x)^n}{x} dx \to 0
    \end{align*}
    and therefore by Theorem \ref{Beta Sum Formula} we have as $n \to \infty$ that
    \begin{align*}
        \beta_{2n+1} \sim \sum_{c=0}^\infty \dfrac{1}{(c+1)2^{c+1}} - \sum_{c=0}^\infty \dfrac{1}{(2c+2) 2^{2c+2}} = \sum_{c=1}^\infty \dfrac{1}{(2c+1)2^{2c+1}} = \dfrac{\log(3)}{2}.
    \end{align*}
    Taking into consideration even $h$, it is clear that $\beta_h \to \frac{\log(3)}{2}$ as $h \to \infty$.
\end{proof}

We are now ready to prove the remaining part of Theorem \ref{Limit Values and Inequality}, and therefore also of \ref{Main Theorem}.

\begin{proof}[Proofs of Theorems \ref{Limit Values and Inequality} and \ref{Main Theorem}]
    By Corollaries \ref{Alpha Limit} and \ref{Beta Limit} along with Theorems \ref{Alpha Asymptotic Theorem} and \ref{Beta Asymptotic Theorem}, Theorem \ref{Limit Values and Inequality} (1) is already proven, and so it only remains to show that $\alpha_h > \beta_h$ for all $h \geq 2$.

    To show this inequality, we first consider $\beta_h$. For $h = 2n \geq 2$ even, it is clear from Theorem \ref{Beta Sum Formula} that
    \begin{align*}
        \beta_{2n} < \sum_{c=1}^\infty \frac{1}{(2c+1) 2^{2c+1}} = \frac{\log(3)}{2}.
    \end{align*}
    To study $\beta_{2n+1}$ for $n \geq 1$, we begin by observing that $\beta_3 = \log(2) - \frac 18 > \frac{\log(3)}{2}$. For $n \geq 2$, we have by Theorem \ref{Beta Sum Formula} that
    \begin{align*}
        \beta_{2n+1} - \beta_{2n-1} &= \dfrac{1}{n 2^n} - \dfrac{1}{n 2^{2n+1}} + \int_{1/2}^1 \dfrac{(1-x)^n - (1-x)^{n-1}}{x} dx \\ &= \dfrac{1}{n 2^n} - \dfrac{1}{n 2^{2n+1}} - \int_{1/2}^1 (1-x)^{n-1} dx \\ &= - \dfrac{1}{n 2^{2n+1}}.
    \end{align*}
    Thus, $\beta_h \leq \log(2) - \frac 18$ for all $h \geq 2$. We now consider a similar study of $\alpha_h$. Recall from \eqref{alpha_h H formula} that $\alpha_h = \frac{\widetilde \alpha_h + \widetilde\alpha_{h-1}}{2}$ for
    \begin{align*}
        \widetilde\alpha_h = \begin{cases} H_h - H_{\frac h2} & \text{ if } 2|h, \\ H_h - H_{\frac{h-1}{2}} & \text{ if } 2 \centernot | h. \end{cases}
    \end{align*}
    Therefore, we have
    \begin{align*}
        \alpha_{h+1} - \alpha_h = \dfrac{\widetilde\alpha_{h+1} -\widetilde\alpha_{h-1}}{2}.
    \end{align*}
    Now, for $h$ odd we have
    \begin{align*}
        \widetilde\alpha_{h+1} - \widetilde\alpha_{h-1} &= \lp H_{h+1} - H_{\frac{h+1}{2}} \rp - \lp H_{h-1} - H_{\frac{h-1}{2}} \rp \\ &= \lp H_{h+1} - H_{h-1} \rp - \lp H_{\frac{h+1}{2}} - H_{\frac{h-1}{2}} \rp = \dfrac{1}{h} - \dfrac{1}{h+1} > 0,
    \end{align*}
    and likewise $\widetilde\alpha_{h+1} - \widetilde\alpha_{h-1} > 0$ for $h$ even as well. Thus, $\alpha_h$ is an increasing function of $h$. Now, we have from \cite{BBCFW} that $\alpha_2 = \frac 34 > \frac 12 = \beta_2$ and $\alpha_3 = \frac 23 > \log(2) - \frac 18 = \beta_3$. Therefore, for $h \geq 4$ we have $\alpha_h > \alpha_3 > \beta_3 > \beta_h$, which completes the proof.
\end{proof}

\section{Final Remarks} \label{Discussion}

\subsection{Frequency of hooks in the rows of partitions}

Probabilistic features of partitions are of great interest, and asymptotic formulas derived from the circle method are very useful for studying such questions \cite{EL, F, NR}. We briefly give an overview of some statistical corollaries which can be derived from our results when combined with other known asymptotics in the literature. To state these results, we let $d(n)$ be the number of partitions into odd parts (or into distinct parts). We first give the average number of hooks equal to any $h \geq 1$ among the partitions of $n$ into odd parts.

\begin{corollary} \label{Averages}
    Let $\mathrm{avg}_{\mathcal L}(h;n)$ be the average number of hooks of length $h$ among the partitions of size $n$ in the collection $\mathcal L$. Then we have
    \begin{align*}
        \mathrm{avg}_{\mathcal O}\lp h;n \rp \sim \dfrac{6 \alpha_h}{\pi} \sqrt{\dfrac{n}{3}}, \quad \mathrm{avg}_{\mathcal D}\lp h;n \rp \sim \dfrac{6 \beta_h}{\pi} \sqrt{\dfrac{n}{3}}.
    \end{align*}
    as $n \to \infty$.
\end{corollary}

\begin{proof}
    It is a well-known consequence of the circle method (and indeed can be proven from \eqref{Odd GF Asymptotic} and Proposition \ref{WrightCircleMethod}) that
    \begin{align}
        p_{\mathcal O}(n) = p_{\mathcal D}(n) \sim \dfrac{3^{3/4}}{12 n^{3/4}} e^{\pi \sqrt{\frac{n}{3}}}.
    \end{align}
    Note that the average number of hooks of length $h$ in a partition of $n$ into odd parts or distinct parts is given by $a_h(n)/p_{\mathcal O}(n)$, $b_h(n)/p_{\mathcal D}(n)$, respectively. This completes the proof along with Theorems \ref{Alpha Asymptotic Theorem} and \ref{Beta Asymptotic Theorem}.
\end{proof}

There is another natural probabilistic question which can now be answered about hooks in these restricted classes of partitions, namely the probability that a uniformly selected part from this class of partitions has a hook of that length in the corresponding row of the Ferrers diagram.

\begin{corollary} \label{Parts Probabilities}
    Let $\mathrm{prob}_{\mathcal L}\lp h; n \rp$ denote the probability that a randomly selected row from among the partitions in $\mathcal L$ of size $n$ has a hook of length $h$. Then we have
    \begin{align*}
        \mathrm{prob}_{\mathcal O}\lp h;n \rp \sim \dfrac{4\alpha_h}{\log(n)}, \quad \mathrm{prob}_{\mathcal D}\lp h;n \rp \sim \dfrac{\beta_h}{\log(2)}
    \end{align*}
    as $n \to \infty$.
\end{corollary}

\begin{proof}
    Note that by construction of the diagram, each row can have at most one hook of length $h$ in each row. Therefore, if we let $d_{\mathcal O}(n)$ and $d_{\mathcal D}(n)$ denote the number of parts among all partitions into odd parts or distinct parts, respectively, it follows that the desired probabilities are given by $a_h(n)/d_{\mathcal O}(n)$ and $b_h(n)/d_{\mathcal D}(n)$, respectively. From the main results of \cite{C,JO2} we obtain
    \begin{align*}
        d_{\mathcal O}(n) \sim \dfrac{3^{\frac 14} \log(n)}{8 \pi n^{\frac 14}} e^{\pi \sqrt{\frac{n}{3}}}, \quad d_{\mathcal D}(n) \sim \dfrac{3^{\frac 14} \log(2)}{2\pi n^{\frac 14}} e^{\pi \sqrt{\frac{n}{3}}}
    \end{align*}
    as $n \to \infty$. The result follows from Theorems \ref{Alpha Asymptotic Theorem} and \ref{Beta Asymptotic Theorem}.
\end{proof}

The results of this section together give a detailed comparison between the properties of hook numbers in these two classes of partitions. From Corollary \ref{Averages}, we reaffirm that the total number of hooks of length $h \geq 2$ is larger for $\mathcal O$ than for $\mathcal D$, and this inequality flips for $h = 1$. However, the perspective of individual parts tells a different story. As $n \to \infty$, we see that that the ``average row" of a partition into odd parts has a vanishing probability of containing a hook of length $h$, while that same probability for the number of partitions into distinct parts is the positive number $\frac{\beta_h}{\log(2)}$. By Corollary \ref{Beta Limit}, this constant approaches $\frac{\log(3)}{\log(4)} \approx 0.7924$. Thus, most rows of partitions into distinct parts have a hook of any given length $h \geq 1$. This discrepancy is reflected by the fact that partitions into odd parts have many more parts than partitions into distinct parts do on average, as is seen from the asymptotics for $d_{\mathcal O}(n)$ and $d_{\mathcal D}(n)$. It would be very interesting to more deeply study these probabilistic features of hooks in partitions, as is done for example in \cite{GOT} for hook numbers which are divisible by $h$ in unrestricted partitions.

\subsection{Open problems and questions}

The motivating question of this study can be greatly generalized; the underlying concept is to understand using the circle method (or other methods) how combinatorial statistics on partitions behave on different subfamilies of partitions. Such studies can be immediately generalized to other related combinatorial objects, such as unimodal sequences or more general integer compositions.

It would also be natural to study such questions for other combinatorial statistics on standard integer partitions, such as ranks and cranks of partitions. Our result does not immediately give an asymptotic count or even inequality for these counting functions, but a sieving argument such as that in \cite{EL} might be useful here. Alternatively, one might add together generating functions $\tilde{a}_{th}(q), \tilde{b}_{th}(q)$ over the values $t \geq 1$ and perform an analysis parallel to that of this paper.

Finally, it would be interesting to pursue variations of this problem for other families of partitions which lie in bijection with one another. As there remain conjectures of \cite{BBCFW} which are not resolved in this aspect, we focus on this case. The primary discussions not addressed here involve the functions $a_h^*(n)$ and $b_h^*(n)$, which count the number of hooks of length $h$ among self-conjugate partitions and partitions into distinct odd parts, respectively. In light of our main theorems and the data presented in \cite{BBCFW}, we present the following conjecture in this setting.

\begin{conjecture}
    For $h \geq 2$, there is a constant $\gamma_h^* > 1$ such that $a^*_h(n)/b_h^*(n) \to \gamma_h^*$ as $n \to \infty$.
\end{conjecture}

It would also be interesting if $\gamma_h^*$ itself had a limit as $h \to \infty$, but we do not speculate on this here. The function $a_h^*(n)$ is of additional interest, as it seems to have significant nontrivial divisibilities. We therefore restate here a conjecture of \cite{BBCFW}.

\begin{conjecture}
    For each $n \geq 0$ and $m \geq 1$, we have
    \begin{align*}
        a^*_{2m}(n) \equiv 0 \pmod{2m}.
    \end{align*}
\end{conjecture}



\begin{thebibliography}{99}
\bibitem{Andrews} G. Andrews. \textit{The Theory of Partitions}, Cambridge University Press, no.\@ 2 (1998).	

\bibitem{A} G. Andrews. \emph{Euler's partition identity and two problems of George Beck}. Math. Student, 86(1-2):115--119, 2017.

\bibitem{AS} A. Ayyer and S. Sinha. \emph{The size of $t$-cores and hook lengths of random cells in partitions}. Ann. Appl. Probab. {\bf 33} (1), 85--106 (2023).

\bibitem{BBCFW} C. Ballantine, H. Burson, W. Craig, A. Folsom, and B. Wen. \emph{Hook length biases and general linear partition inequalities}. Res. Math. Sci. {\bf 10}, 41 (2023).

\bibitem{BH} C. Bessenrodt and G.-N. Han. \emph{Symmetry distribution between hook length and part length for partitions}. Discrete Math. {\bf 309} (2009), pp. 6070--6073.

\bibitem{BCM} K. Bringmann, W. Craig, and J. Males. \emph{Asymptotics for $d$-fold partition diamonds and related infinite products}. Preprint, arXiv:2311.11805.

\bibitem{BJU} K. Bringmann, C. Jennings-Shaffer, and K. Mahlburg. \emph{On a Tauberian theorem of Ingham and Euler--Maclaurin summation}. Ramanujan J. {\bf 61}, 55--86 (2023). 

\bibitem{BOW} K. Bringmann, K. Ono, and I. Wagner. \emph{Eichler integrals of Eisenstein series as $q$-brackets of weighted $t$-hook functions on partitions}. Ramanujan J. {\bf 61}, 279--293 (2023).

\bibitem{CKNS} H. Cho, B. Kim, H. Nam, and J. Sohn. \emph{A survey on $t$-cores partitions}. Hardy-Ramanujan J. {\bf 44}, 81--101 (2021).

\bibitem{C} W. Craig. \emph{On the number of parts in congruence classes for partitions into distinct parts}. Res. Number Theory, {\bf 8} (3):52, 2022.

\bibitem{EL} P. Erd\H{o}s and J. Lehner. \emph{The distribution of the number of summands in the partitions of a positive integer}. Duke Math. J. {\bf 8}, 335--345 (1941).

\bibitem{F} B. Fristedt, \emph{The structure of random large partitions of integers}, Trans. Am. Math. Soc. {\bf 337}, 703--735 (1993).

\bibitem{GKS} F. Garvan, D. Kim, and D. Stanton. \emph{Cranks and $t$-cores}. Invent. math. {\bf 101} 1, 1--18 (1990).

\bibitem{Glaisher} J.W.L. Glaisher. \emph{On the partition of numbers into odd parts}. The Messenger of Mathematics {\bf 12}, 131--136 (1883).

\bibitem{GKW} P.J. Grabner, A. Knopfmacher, and S. Wagner. \emph{A General Asymptotic Scheme for the Analysis of Partition Statistics}. Comb. Probab. Comput., Volume 23, Issue 6: Honouring the Memory of Philippe Flajolet - Part 2, November 2014, pp. 1057--1086.

\bibitem{GO} A. Granville, K. Ono. \emph{Defect zero blocks for finite simple groups}. Trans. Am. Math. Soc. {\bf 348}, 1 (1996).

\bibitem{GOT} M. Griffin, K. Ono, and W.-L. Tsai. \emph{Distributions of hook lengths in integer partitions}. Proc. Am. Math. Soc., Series B. In Press.

\bibitem{H} G.-N. Han. \emph{The Nekrasov-Okounkov hook length formula: refinement, elementary proof, extension and applications}. Ann. Inst. Fourier (Grenoble), 60(1):1--29, 2010.

\bibitem{JO} F. Jackson and M. Otgonbayar. \emph{Biases among congruence classes for parts in $k$-regular partitions}. arXiv preprint, arXiv:2207.04352, 2022.

\bibitem{JO2} F. Jackson and M. Otgonbayar. \emph{Unexpected biases between congruence classes for parts in $k$-indivisible partitions}. J. Number Theory \textbf{248}, 310--342 (2023).

\bibitem{NO} N. A. Nekrasov and A. Okounkov. \emph{Seiberg-Witten theory and random partitions}. In \emph{The unity of mathematics}, Volume 244 of Progr. Math., pages 525--596. Birkh\"{a}user Boston, Boston, MA, 2006.

\bibitem{NR} H.T. Ngo and R. Rhoades. \emph{Integer partitions, probabilities and quantum modular forms}. Res. Math. Sci. {\bf 4}, Article number: 17 (2017).

\bibitem{OS} K. Ono and L. Sze. \emph{4-core partitions and class numbers}. Acta Arith. {\bf 80} 3, 249--272 (1997).

\end{thebibliography}
\end{document}